\documentclass[a4paper]{amsart}
\usepackage{color}
\usepackage[latin1]{inputenc}
\usepackage[T1]{fontenc}
\usepackage{amsfonts}
\usepackage{amssymb}
\usepackage{amsmath}
\usepackage{amsthm}
\usepackage{graphicx,type1cm,eso-pic,color}
\usepackage{pstricks,pst-plot,pstricks-add}
\usepackage{float}
\newtheorem{theorem}{Theorem}[section]
\newtheorem{lemma}[theorem]{Lemma}
\newtheorem{proposition}[theorem]{Proposition}
\newtheorem{corollary}[theorem]{Corollary}
\newtheorem{definition}[theorem]{Definition}
\newtheorem{assumption}[theorem]{Assumption}
\begin{document}
\setlength\arraycolsep{2pt}
\title{A generalized mean-reverting equation and applications}
\author{Nicolas MARIE}
\address{Laboratoire Modal'X, Universit\'e Paris-Ouest, 200 Avenue de la R\'epublique, 92000 Nanterre, France}
\email{nmarie@u-paris10.fr}
\keywords{Stochastic differential equations, rough paths, large deviation principle, mean-reversion, Gaussian processes}
\date{}
\maketitle
%

% Abstract.

%
\begin{abstract}
Consider a mean-reverting equation, generalized in the sense it is driven by a $1$-dimensional centered Gaussian process with H\"older continuous paths on $[0,T]$ ($T > 0$). Taking that equation in rough paths sense only gives local existence of the solution because the non-explosion condition is not satisfied in general. Under natural assumptions, by using specific methods, we show the global existence and uniqueness of the solution, its integrability, the continuity and differentiability of the associated It\^o map, and we provide an $L^p$-converging approximation with a rate of convergence ($p\geqslant 1$). The regularity of the It\^o map ensures a large deviation principle, and the existence of a density with respect to Lebesgue's measure, for the solution of that generalized mean-reverting equation. Finally, we study a generalized mean-reverting pharmacokinetic model.
\end{abstract}
\tableofcontents
\noindent
\textbf{MSC2010 :} 60H10.
\\
\\
\textbf{Acknowledgements.} Many thanks to my Ph.D. supervisor Laure Coutin for her precious help and advices. This work was supported by A.N.R. Masterie.
%

% Section : Introduction.

%
\section{Introduction}
\noindent
Let $W$ be a $1$-dimensional centered Gaussian process with $\alpha$-H\"older continuous paths on $[0,T]$ ($T > 0$ and $\alpha\in ]0,1]$).
\\
\\
Consider the stochastic differential equation (SDE) :
\begin{equation}\label{EDS3}
X_t  =
x_0 +
\int_{0}^{t}
\left(a - bX_u\right)du +
\sigma\int_{0}^{t}
X_{u}^{\beta}dW_u
\textrm{ ; }
t\in [0,T]
\end{equation}
where, $x_0 > 0$ is a deterministic initial condition, $a,b,\sigma\geqslant 0$ are deterministic constants and $\beta$ satisfies the following assumption :
%

% Assumption : Condition on $\beta$ for existence.

%
\begin{assumption}\label{Eass}
The exponent $\beta$ satisfies : $\beta\in ]1-\alpha,1]$.
\end{assumption}
\noindent
When the driving signal is a standard Brownian motion, equation (\ref{EDS3}) taken in the sense of It\^o, is used in many applications. For example, it is studied and applied in finance by J-P. Fouque et al. in \cite{FFK10} for $\beta\in [1/2,1[$. The cornerstone of their approach is the Markov property of diffusion processes. In particular, their proof of the global existence and uniqueness of the solution at Appendix A involves S. Karlin and H.M. Taylor \cite{KT81}, Lemma 6.1(ii). Still for $\beta\in [1/2,1[$, the convergence of the Euler approximation is proved by X. Mao et al. in \cite{MTY06} and \cite{WMC08}. For $\beta\geqslant 1$, equation (\ref{EDS3}) is studied by F. Wu et al. in \cite{WMC08}. Recently, in \cite{NTD11}, N. Tien Dung got an expression and shown the Malliavin's differentiability of a class of fractional geometric mean-reverting processes.
\\
\\
Equation (\ref{EDS3}) is a generalization of the mean-reverting equation. In this paper, we study various properties of (\ref{EDS3}) by taking it in the sense of rough paths (cf. T. Lyons and Z. Qian \cite{LQ02}). Note that Doss-Sussman's method could also be used since (\ref{EDS3}) is a $1$-dimensional equation (cf. H. Doss \cite{DOSS76} and H.J. Sussman \cite{SUSS78}). A priori, even in these senses, equation (\ref{EDS3}) admits only a local solution because it doesn't satisfy the non-explosion condition of \cite{FV08}, Exercice 10.56. 
\\
\\
At Section 2, we state useful results on rough differential equations (RDEs) and Gaussian rough paths coming from P. Friz and N. Victoir \cite{FV08}. Section 3 is devoted to study deterministic properties of (\ref{EDS3}). We show existence and uniqueness of the solution for equation (\ref{EDS3}), provide an explicit upper-bound for that solution and study the continuity and differentiability of the associated It\^o map. We also provide a converging approximation with a rate of convergence. Section 4 is devoted to study probabilistic properties of (\ref{EDS3}) ; properties of the solution's distribution, various integrability results, a large deviation principle and the existence of a density with respect to Lebesgue's measure on $(\mathbb R,\mathcal B(\mathbb R))$ for the solution of (\ref{EDS3}). Finally, at Section 5, we study a pharmacokinetic model based on a particular generalized mean-reverting (M-R) equation (inspired by K. Kalogeropoulos et al. \cite{KDP08}).
%

% Section : Rough differential equations and Gaussian rough paths.

%
\section{Rough differential equations and Gaussian rough paths}
\noindent
Essentially inspired by P. Friz and N. Victoir \cite{FV08}, this section provides useful definitions and results on RDEs and Gaussian rough paths.
\\
\\
In a sake of completeness, results on rough differential equations are stated in the multidimensional case.
\\
\\
In the sequel, $\|.\|$ denotes the euclidean norm on $\mathbb R^d$ and $\|.\|_{\mathcal M}$ the usual norm on $\mathcal M_d(\mathbb R)$ $(d\in\mathbb N^*)$.
\\
\\
Consider $D_T$ the set of subdivisions for $[0,T]$ and
\begin{displaymath}
\Delta_T =
\left\{(s,t)\in\mathbb R_{+}^{2} :
0\leqslant s < t\leqslant T\right\}.
\end{displaymath}
Let $T^N(\mathbb R^d)$ be the step-$N$ tensor algebra over $\mathbb R^d$ ($N\in\mathbb N^*$) :
\begin{displaymath}
T^N\left(\mathbb R^d\right) =
\bigoplus_{i = 0}^{N}
\left(\mathbb R^d\right)^{\otimes i}.
\end{displaymath}
For $i = 1,\dots,N$, $(\mathbb R^d)^{\otimes i}$ is equipped with its euclidean norm $\|.\|_i$, $(\mathbb R^d)^{\otimes 0} = \mathbb R$ and the canonical projection on $(\mathbb R^d)^{\otimes i}$ for any $Y\in T^N(\mathbb R^d)$ is denoted by $Y^i$.
\\
\\
First, let's remind definitions of $p$-variation and $\alpha$-H\"older norms ($p\geqslant 1$ and $\alpha\in [0,1]$) :
%

% Definition : p-variation and $\alpha$-Hlder topologies.

%
\begin{definition}\label{VARtopo}
Consider $y : [0,T]\rightarrow\mathbb R^d$ :
\begin{enumerate}
 \item The function $y$ has finite $p$-variation if and only if,
 \begin{displaymath}
 \|y\|_{p\textrm{-var};T} =
 \sup_{D = \{r_k\}\in D_T}\left(
 \sum_{k = 1}^{|D| - 1}\|y_{r_{k + 1}} - y_{r_k}\|^p\right)^{1/p} < \infty.
 \end{displaymath}
 \item The function $y$ is $\alpha$-H\"older continuous if and only if,
 \begin{displaymath}
 \|y\|_{\alpha\textrm{-H\"ol};T} =
 \sup_{(s,t)\in\Delta_T}
 \frac{\|y_t - y_s\|}{|t - s|^{\alpha}} < \infty. 
 \end{displaymath}
\end{enumerate}
\end{definition}
\noindent
In the sequel, the space of continuous functions with finite $p$-variation will be denoted by :
\begin{displaymath}
C^{p\textrm{-var}}\left([0,T];\mathbb R^d\right).
\end{displaymath}
The space of $\alpha$-H\"older continuous functions will be denoted by :
\begin{displaymath}
C^{\alpha\textrm{-H\"ol}}\left([0,T];\mathbb R^d\right).
\end{displaymath}
If it is not specified, these spaces will always be equipped with norms $\|.\|_{p\textrm{-var};T}$ and $\|.\|_{\alpha\textrm{-H\"ol};T}$ respectively.
\\
\\
\textbf{Remark.} Note that :
\begin{displaymath}
C^{\alpha\textrm{-H\"ol}}\left([0,T];\mathbb R^d\right)
\subset
C^{1/\alpha\textrm{-var}}\left([0,T];\mathbb R^d\right).
\end{displaymath}
%

% Definition : Signature.

%
\begin{definition}\label{signature}
Let $y : [0,T]\rightarrow\mathbb R^d$ be a continuous function of finite $1$-variation. The step-$N$ signature of $y$ is the functional $S_N(y) : \Delta_T\rightarrow T^N(\mathbb R^d)$ such that for every $(s,t)\in\Delta_T$ and $i = 1,\dots,N$,
\begin{displaymath}
S_{N;s,t}^{0}(y) = 1
\textrm{ and }
S_{N;s,t}^{i}(y) =
\int_{s < r_1 < r_2 < \dots < r_i < t}
dy_{r_1}\otimes\dots\otimes dy_{r_i}.
\end{displaymath}
Moreover,
\begin{displaymath}
G^N(\mathbb R^d) =
\left\{S_{N;0,T}(y) ; y\in C^{1\textrm{-var}}([0,T];\mathbb R^d)\right\}
\end{displaymath}
is the step-$N$ free nilpotent group over $\mathbb R^d$.
\end{definition}
%

% Definition : Norm in $p$-variation for rough paths.

%
\begin{definition}\label{Npvar}
A map $Y : \Delta_T\rightarrow G^N(\mathbb R^d)$ is of finite $p$-variation if and only if,
\begin{displaymath}
\|Y\|_{p\textrm{-var};T} =
\sup_{
D = \left\{r_k\right\}\in D_T}
\left(\sum_{k = 1}^{|D| - 1}
\|Y_{r_k,r_{k + 1}}\|_{\mathcal C}^{p}\right)^{1/p} < \infty
\end{displaymath}
where, $\|.\|_{\mathcal C}$ is the Carnot-Caratheodory's norm such that for every $g\in G^N(\mathbb R^d)$,
\begin{displaymath}
\|g\|_{\mathcal C} =
\inf\left\{
\mathbf{length}(y) ;
y\in
C^{1\textrm{-var}}([0,T];\mathbb R^d)
\textrm{ and }
S_{N;0,T}(y) = g\right\}.
\end{displaymath}
\end{definition}
\noindent
In the sequel, the space of continuous functions from $\Delta_T$ into $G^N(\mathbb R^d)$ with finite $p$-variation will be denoted by :
\begin{displaymath}
C^{p\textrm{-var}}([0,T];G^N(\mathbb R^d)).
\end{displaymath}
If it is not specified, that space will always be equipped with $\|.\|_{p\textrm{-var};T}$.
\\
\\
Let's define the Lipschitz regularity in the sense of Stein :
%

% Definition : $\gamma$-Lip vector fields.

%
\begin{definition}\label{LIPvect}
Consider $\gamma > 0$. A map $V : \mathbb R^d\rightarrow\mathbb R$ is $\gamma$-Lipschitz (in the sense of Stein) if and only if $V$ is $C^{\lfloor\gamma\rfloor}$ on $\mathbb R^d$, bounded, with bounded derivatives and such that the $\lfloor\gamma\rfloor$-th derivative of $V$ is $\{\gamma\}$-H\"older continuous ($\lfloor\gamma\rfloor$ is the largest integer strictly smaller than $\gamma$ and $\{\gamma\} = \gamma - \lfloor\gamma\rfloor$).
\\
\\
The least bound is denoted by $\|V\|_{\textrm{lip}^{\gamma}}$. The map $\|.\|_{\textrm{lip}^{\gamma}}$ is a norm on the vector space of collections of $\gamma$-Lipschitz vector fields on $\mathbb R^d$, denoted by $\textrm{Lip}^{\gamma}(\mathbb R^d)$.
\end{definition}
\noindent
In the sequel, $\textrm{Lip}^{\gamma}(\mathbb R^d)$ will always be equipped with $\|.\|_{\textrm{lip}^{\gamma}}$.
\\
\\
Let $w : [0,T]\rightarrow\mathbb R^d$ be a continuous function of finite $p$-variation such that a geometric $p$-rough path $\mathbb W$ exists over it. In other words, there exists an approximating sequence $(w^n,n\in\mathbb N)$ of functions of finite $1$-variation such that :
\begin{displaymath}
\lim_{n\rightarrow\infty}
d_{p\textrm{-var};T}\left[
S_{[p]}\left(w^n\right);\mathbb W\right] = 0.
\end{displaymath}
When $d = 1$, a \textit{natural} geometric $p$-rough path $\mathbb W$ over it is defined by :
\begin{equation}\label{GRP1dim}
\forall (s,t)\in\Delta_T
\textrm{, }
\mathbb W_{s,t} =
\left(1,w_t - w_s,\dots,
\frac{(w_t - w_s)^{[p]}}{[p]!}\right).
\end{equation}
We remind that if $V = (V_1,\dots,V_d)$ is a collection of Lipschitz continuous vector fields on $\mathbb R^d$, the ordinary differential equation $dy = V(y)dw^n$, with initial condition $y_0\in\mathbb R^d$, admits a unique solution.
\\
\\
That solution is denoted by $\pi_V(0,y_0;w^n)$.
\\
\\
Rigorously, a RDE's solution is defined as follow (cf. \cite{FV08}, Definition 10.17) :
%

% Definition : RDE's solution.

%
\begin{definition}\label{RDEsol}
A continuous function $y : [0,T]\rightarrow\mathbb R^d$ is a solution of $dy = V(y)d\mathbb W$ with initial condition $y_0\in\mathbb R^d$ if and only if,
\begin{displaymath}
\lim_{n\rightarrow\infty}
\left\|\pi_V(0,y_0;w^n) - y\right\|_{\infty;T} = 0
\end{displaymath}
where, $\|.\|_{\infty;T}$ is the uniform norm on $[0,T]$. If there exists a unique solution, it is denoted by $\pi_V(0,y_0;\mathbb W)$.
\end{definition}
%

% Theorem : Existence and uniqueness of the solution.

%
\begin{theorem}\label{RDEres}
Let $V = (V_1,\dots,V_d)$ be a collection of locally $\gamma$-Lipschitz vector fields on $\mathbb R^d$ ($\gamma > p$) such that : $V$ and $D^{[p]}V$ are respectively globally Lipschitz continuous and $(\gamma - [p])$-H\"older continuous on $\mathbb R^d$. With initial condition $y_0\in\mathbb R^d$, equation $dy = V(y)d\mathbb W$ admits a unique solution $\pi_V(0,y_0;\mathbb W)$.
\end{theorem}
\noindent
For a proof, see P. Friz and N. Victoir \cite{FV08}, Exercice 10.56.
\\
\\
For P. Friz and N. Victoir, the rough integral for a collection of $(\gamma - 1)$-Lipschitz vector fields $V = (V_1,\dots,V_d)$ along $\mathbb W$ is the projection of a particular full RDE's solution (cf. \cite{FV08}, Definition 10.34 for full RDEs) : $d\mathbb X = \Phi(\mathbb X)d\mathbb W$ where,
\begin{displaymath}
\forall i = 1,\dots,d\textrm{, }
\forall a,w\in\mathbb R^d\textrm{, }
\Phi_i(w,a) = (e_i,V_i(w))
\end{displaymath}
and $(e_1,\dots,e_d)$ is the canonical basis of $\mathbb R^d$.
\\
\\
In particular, if $y : [0,T]\rightarrow\mathcal M_d(\mathbb R)$ and $z : [0,T]\rightarrow\mathbb R^d$ are two continuous functions, respectively of finite $p$-variation and finite $q$-variation with $1/p + 1/q > 1$, the Young integral of $y$ with respect to $z$ is denoted by $\mathcal Y(y,z)$.
\\
\\
\textbf{Remark.} We are not developing the notion of full RDE in that paper because it is not useful in the sequel. As mentioned above, the reader can refer to \cite{FV08}, Definition 10.34 for details.
\\
\\
For a proof of the following change of variable formula for geometric rough paths, cf. \cite{COUTIN11}, Theorem 53 :
%

% Theorem : Change of variable formula.

%
\begin{theorem}\label{CVformula}
Let $\Phi$ be a collection of $\gamma$-Lipschitz vector fields on $\mathbb R^d$ ($\gamma > p$) and let $\mathbb W$ be a geometric $p$-rough path. Then,
\begin{displaymath}
\forall (s,t)\in\Delta_T
\textrm{, }
\Phi\left(w_t\right) - \Phi\left(w_s\right) =
\left[
\int D\Phi(\mathbb W)d\mathbb W\right]_{s,t}^{1}.
\end{displaymath}
\end{theorem}
\noindent
Now, let state some results on $1$-dimensional Gaussian rough paths :
\\
\\
Consider a stochastic process $W$ defined on $[0,T]$ and satisfying the following assumption :
%

% Assumption : Condition on $W$.

%
\begin{assumption}\label{HYPW}
$W$ is a $1$-dimensional centered Gaussian process with $\alpha$-H\"older continuous paths on $[0,T]$ ($\alpha\in ]0,1]$).
\end{assumption}
\noindent
In the sequel, we work on the probability space $(\Omega,\mathcal A,\mathbb P)$ where $\Omega = C^0([0,T];\mathbb R)$, $\mathcal A$ is the $\sigma$-algebra generated by cylinder sets and $\mathbb P$ is the probability measure induced by $W$ on $(\Omega,\mathcal A)$.
\\
\\
\textbf{Remark.} Since $W$ is a $1$-dimensional Gaussian process, the \textit{natural} geometric $1/\alpha$-rough path over it defined by (\ref{GRP1dim}) is matching with the enhanced Gaussian process for $W$ provided by P. Friz and N. Victoir at \cite{FV08}, Theorem 15.33 in the multidimensional case.
\\
\\
Finally, Cameron-Martin's space of $W$ is given by :
\begin{displaymath}
\mathcal H_{W}^{1} =
\left\{
h\in C^0([0,T];\mathbb R) :
\exists Z\in\mathcal A_W
\textrm{ s.t. }
\forall t\in [0,T]\textrm{, }
h_t =
\mathbb E(W_tZ)
\right\}
\end{displaymath}
with
\begin{displaymath}
\mathcal A_W =
\overline{\textrm{span}
\left\{W_t;
t\in [0,T]\right\}}^{L^2}.
\end{displaymath}
Let $\langle .,.\rangle_{\mathcal H_{W}^{1}}$ be the map defined on $\mathcal H_{W}^{1}\times\mathcal H_{W}^{1}$ by :
\begin{displaymath}
\langle h,\tilde h\rangle_{\mathcal H_{W}^{1}} =
\mathbb E(Z\tilde Z)
\end{displaymath}
where,
\begin{displaymath}
\forall t\in [0,T]\textrm{, }
h_t =
\mathbb E(W_tZ)
\textrm{ and }
\tilde h_t =
\mathbb E(W_t\tilde Z)
\end{displaymath}
with $Z,\tilde Z\in\mathcal A_W$.
\\
\\
That map is a scalar product on $\mathcal H_{W}^{1}$ and, equipped with it, $\mathcal H_{W}^{1}$ is a Hilbert space.
\\
\\
The triplet $(\Omega,\mathcal H_{W}^{1},\mathbb P)$ is called an abstract Wiener space (cf. M. Ledoux \cite{LED94}).
%

% Proposition : Bouleau-Hirsch criterion.

%
\begin{proposition}\label{BHcondition}
For $d = 1$, consider a random variable $F :\Omega\rightarrow\mathbb R$, continuously $\mathcal H_{W}^{1}$-differentiable (i.e. $h\mapsto F(\omega + h)$ is continuously differentiable from $\mathcal H_{W}^{1}$ into $\mathbb R$, for almost every $\omega\in\Omega$).
\\
\\
If $F$ satisfies Bouleau-Hirsch's condition (i.e. $|D_hF| > 0$ a.s. for at least one $h\in\mathcal H_{W}^{1}$ such that $h\not= 0$, where :
\begin{displaymath}
(D_{\eta}F)(\omega) =
\left.\frac{\partial}{\partial\varepsilon}
F(\omega +\varepsilon\eta)
\right|_{\varepsilon = 0}
\textrm{, }
\forall\eta\in\mathcal H_{W}^{1}),
\end{displaymath}
then $F$ admits a density with respect to Lebesgue's measure on $(\mathbb R,\mathcal B(\mathbb R))$.
\end{proposition}
\noindent
\textbf{Remarks :}
\begin{enumerate}
 \item Classically, Bouleau-Hirsch's condition is not stated that way and involves Malliavin calculus framework. Consider the Malliavin derivative operator $\mathbf D$ (cf. D. Nualart \cite{NUAL05}, Section 1.2), the reproducing kernel Hilbert space $\mathcal H_W$ of the Gaussian process $W$ (cf. J. Neveu \cite{NEVEU68}), and the canonical isometry $I$ from $\mathcal H_W$ into $\mathcal H_{W}^{1}$ defined for example at N. Marie \cite{MARIE11}, Section 3.1. Bouleau-Hirsch's condition for $d = 1$ is $\|\mathbf DF\|_{\mathcal H}^{2} > 0$.
 \\
 \\
 On one hand, by Cauchy-Schwarz's inequality, it is sufficient to show that there exists $h\in\mathcal H_{W}^{1}$ satisfying $h\not= 0$ and $|\langle\mathbf DF, I^{-1}(h)\rangle_{\mathcal H}| > 0$. On the other hand, with Malliavin calculus methods, one can easily show that $\langle\mathbf DF,I^{-1}(h)\rangle_{\mathcal H} = D_hF$.
 \item About Bouleau-Hirsch's criterion for $d\geqslant 1$, please refer to \cite{NUAL05}, Theorem 2.1.2.
\end{enumerate}
%

% Section : Deterministic properties of the generalized mean-reverting equation.

%
\section{Deterministic properties of the generalized mean-reverting equation}
\noindent
In this section, we show existence and uniqueness of the solution for equation (\ref{EDS3}), provide an explicit upper-bound for that solution and, study the continuity and differentiability of the associated It\^o map. We also provide a converging approximation for equation (\ref{EDS3}).
\\
\\
Consider a function $w : [0,T]\rightarrow\mathbb R$ satisfying the following assumption :
%

% Assumption : Conditions on the driving signal.

%
\begin{assumption}\label{HYPw}
The function $w$ is $\alpha$-H\"older continuous ($\alpha\in ]0,1]$).
\end{assumption}
\noindent
Let $\mathbb W$ be the \textit{natural} geometric $1/\alpha$-rough path over $w$ defined by (\ref{GRP1dim}). Then, we put $\mathcal W = S_{[1/\alpha]}(\textrm{Id}_{[0,T]}\oplus\mathbb W)$, which is a geometric $1/\alpha$-rough path over
\begin{displaymath}
t\in [0,T]\longmapsto
\left(t,w_t\right)
\end{displaymath}
by \cite{FV08}, Theorem 9.26.
\\
\\
\textbf{Remark.} For a rigorous construction of Young pairing, the reader can refer to \cite{FV08}, Section 9.4.
\\
\\
Then, consider the rough differential equation :
\begin{equation}\label{DETmre}
dx = V(x)d\mathcal W
\textrm{ with initial condition $x_0\in\mathbb R$,}
\end{equation}
where $V$ is the map defined on $\mathbb R_+$ by :
\begin{displaymath}
\forall x\in\mathbb R_+\textrm{, }
\forall t,w\in\mathbb R\textrm{, }
V(x).(t,w) =
(a - bx)t +
\sigma x^{\beta}w.
\end{displaymath}
For technical reasons, we introduce another equation :
\begin{equation}\label{EDS4}
y_t =
y_0 +
a(1-\beta)
\int_{0}^{t}
y_{s}^{-\gamma}e^{bs}ds +
\tilde w_t
\textrm{ ; }
t\in [0,T]
\textrm{, }
y_0 > 0
\end{equation}
where, $\gamma = \frac{\beta}{1-\beta}$ and
\begin{displaymath}
\tilde w_t =
\int_{0}^{t}\vartheta_sdw_s
\textrm{ with }
\vartheta_t =
\sigma(1-\beta)e^{b(1-\beta)t}
\end{displaymath}
for every $t\in [0,T]$. The integral is taken in the sense of Young.
\\
\\
The map $u\in [\varepsilon,\infty[\mapsto u^{-\gamma}$ belongs to $C^{\infty}([\varepsilon,\infty[;\mathbb R)$ and is bounded with bounded derivatives on $[\varepsilon,\infty[$ for every $\varepsilon > 0$. Then, equation (\ref{EDS4}) admits a unique solution in the sense of Definition \ref{RDEsol} by applying Theorem \ref{RDEres} up to the time
\begin{displaymath}
\tau_{\varepsilon}^{1} =
\inf\left\{t\in [0,T] :
y_t = \varepsilon\right\}
\textrm{ ; }
\varepsilon\in
]0,y_0],
\end{displaymath}
by assuming that $\inf(\emptyset) = \infty$.
\\
\\
Consider also the time $\tau_{0}^{1} > 0$, such that $\tau_{\varepsilon}^{1}\uparrow\tau_{0}^{1}$ when $\varepsilon\rightarrow 0$.
%

% Existence and uniqueness of the solution.

%
\subsection{Existence and uniqueness of the solution}
As mentioned above, Section 2 ensures that equation (\ref{EDS4}) has, at least locally, a unique solution denoted $y$. At Lemma \ref{EDS3change}, we prove it ensures that equation (\ref{DETmre}) admits also, at least locally, a unique solution (in the sense of Definition \ref{RDEsol}) denoted $x$. In particular, we show that $x = y^{\gamma + 1}e^{-b.}$. At Proposition \ref{MRres}, we prove the global existence of $y$ by using the fact it never hits $0$ on $[0,T]$. These results together ensures the existence and uniqueness of $x$ on $[0,T]$.
%

% Lemma : Change of equation.

%
\begin{lemma}\label{EDS3change}
Consider $y_0 > 0$ and $a,b\geqslant 0$. Under assumptions \ref{Eass} and \ref{HYPw}, up to the time $\tau_{\varepsilon}^{1}$ ($\varepsilon\in ]0,y_0]$), if $y$ is the solution of (\ref{EDS4}) with initial condition $y_0$, then
\begin{displaymath}
x : t\in \left[0,\tau_{\varepsilon}^{1}\right]\longmapsto
x_t = y_{t}^{\gamma + 1}e^{-bt}
\end{displaymath}
is the solution of (\ref{DETmre}) on $[0,\tau_{\varepsilon}^{1}]$, with initial condition $x_0 = y_{0}^{\gamma + 1}$.
\end{lemma}
%

% Proof.

%
\begin{proof}
Consider the solution $y$ of (\ref{EDS4}) on $[0,\tau_{\varepsilon}^{1}]$, with initial condition $y_0 > 0$.
\\
\\
The continuous function $z = ye^{-b(1-\beta).}$ takes its values in $[m_{\varepsilon},M_{\varepsilon}]\subset\mathbb R_{+}^{*}$ on $[0,\tau_{\varepsilon}^{1}]$.
\\
\\
Since $\gamma > 0$, the map $\Phi : u\in [m_{\varepsilon},M_{\varepsilon}]\mapsto u^{\gamma + 1}$ is $C^{\infty}$, bounded and with bounded derivatives.
\\
\\
Then, by applying the change of variable formula (Theorem \ref{CVformula}) to $z$ and to the map $\Phi$ between $0$ and $t\in [0,\tau_{\varepsilon}^{1}]$ :
\begin{eqnarray*}
 x_t & = &
 z_{0}^{\gamma + 1} +
 (\gamma + 1)
 \int_{0}^{t}z_{s}^{\gamma}dz_s\\
 & = &
 y_{0}^{\gamma + 1} +
 \int_{0}^{t}\left(a - bx_s\right)ds +
 \sigma\int_{0}^{t}y_{s}^{\gamma}e^{-b\beta s}dw_s.
\end{eqnarray*}
Since $\gamma = \beta(\gamma + 1)$, in the sense of Definition \ref{RDEsol}, $x$ is the solution of (\ref{DETmre}) on $[0,\tau_{\varepsilon}^{1}]$ with initial condition $x_0 = y_{0}^{\gamma + 1}$.
\end{proof}
%

% Proposition : Existence and uniqueness of the solution.

%
\begin{proposition}\label{MRres}
Under assumptions \ref{Eass} and \ref{HYPw}, for $a > 0$ and $b\geqslant 0$, with initial condition $x_0 > 0$ ; $\tau_{0}^{1} > T$ and then, equation (\ref{DETmre}) admits a unique solution $\tilde\pi_V(0,x_0;w)$ on $[0,T]$, satisfying :
\begin{displaymath}
\tilde\pi_V(0,x_0;w) =
\pi_V(0,x_0;\mathcal W).
\end{displaymath}
Moreover, since $T > 0$ is chosen arbitrarily, that notion of solution extends to $\mathbb R_+$.
\end{proposition}
%

% Proof.

%
\begin{proof}
Suppose that $\tau_{0}^{1}\leqslant T$, put $y_0 = x_{0}^{1-\beta}$ and consider the solution $y$ of (\ref{EDS4}) on $[0,\tau_{\varepsilon}^{1}]$ ($\varepsilon\in ]0,y_0]$), with initial condition $y_0$.
\\
\\
On one hand, note that by definition of $\tau_{\varepsilon}^{1}$ :
\begin{eqnarray*}
 y_{\tau_{\varepsilon}^{1}} - y_t & = & \varepsilon - y_t\textrm{ and}\\
 y_{\tau_{\varepsilon}^{1}} - y_t & = &
 a(1-\beta)
 \int_{t}^{\tau_{\varepsilon}^{1}}y_{s}^{-\gamma}e^{bs}ds +
 \tilde w_{\tau_{\varepsilon}^{1}} -
 \tilde w_t
\end{eqnarray*}
for every $t\in [0,\tau_{\varepsilon}^{1}]$. Then, since $\tau_{\varepsilon}^{1}\uparrow\tau_{0}^{1}$ when $\varepsilon\rightarrow 0$ :
\begin{equation}\label{EGMRres}
y_t +
a(1-\beta)
\int_{t}^{\tau_{0}^{1}}y_{s}^{-\gamma}e^{bs}ds =
\tilde w_t -
\tilde w_{\tau_{0}^{1}}
\end{equation}
for every $t\in [0,\tau_{0}^{1}[$.
\\
\\
Moreover, since $\tilde w$ is the Young integral of $\vartheta\in C^{\infty}([0,T];\mathbb R_+)$ against $w$, and $w$ is $\alpha$-H\"older continuous, $\tilde w$ is also $\alpha$-H\"older continuous (cf. \cite{FV08}, Theorem 6.8).
\\
\\
Together, equality (\ref{EGMRres}) and the $\alpha$-H\"older continuity of $\tilde w$ imply :
\begin{displaymath}
-\|\tilde w\|_{\alpha\textrm{-H\"ol};T}(\tau_{0}^{1} - t)^{\alpha}
\leqslant
y_t +
a(1-\beta)
\int_{t}^{\tau_{0}^{1}}y_{s}^{-\gamma}e^{bs}ds
\leqslant
\|\tilde w\|_{\alpha\textrm{-H\"ol};T}(\tau_{0}^{1} - t)^{\alpha}.
\end{displaymath}
On the other hand, the two terms of that sum are positive. Then,
\begin{eqnarray}
 \label{MAJ1MRres}
 y_t & \leqslant & \|\tilde w\|_{\alpha\textrm{-H\"ol};T}(\tau_{0}^{1} - t)^{\alpha}
 \textrm{ and}\\
 \label{MAJ2MRres}
 a(1-\beta)
 \int_{t}^{\tau_{0}^{1}}y_{s}^{-\gamma}e^{bs}ds
 & \leqslant &
 \|\tilde w\|_{\alpha\textrm{-H\"ol};T}(\tau_{0}^{1} - t)^{\alpha}.
\end{eqnarray}
Since $t\in [0,\tau_{0}^{1}[$ has been chosen arbitrarily, inequality (\ref{MAJ1MRres}) is true for every $s\in [t,\tau_{0}^{1}[$ and implies :
\begin{displaymath}
 y_{s}^{-\gamma}\geqslant
 \|\tilde w\|_{\alpha\textrm{-H\"ol};T}^{-\gamma}\left(\tau_{0}^{1} - s\right)^{-\alpha\gamma}.
\end{displaymath}
So
\begin{eqnarray}
 a(1-\beta)
 \int_{t}^{\tau_{0}^{1}}
 y_{s}^{-\gamma}e^{bs}ds
 & \geqslant &
 a(1-\beta)\|\tilde w\|_{\alpha\textrm{-H\"ol};T}^{-\gamma}
 \int_{t}^{\tau_{0}^{1}}
 (\tau_{0}^{1} - s)^{-\alpha\gamma}e^{bs}ds
 \nonumber\\
 \label{MAJ3MRres}
 & \geqslant &
 \frac{a(1-\beta)}{1-\alpha\gamma}
 \|\tilde w\|_{\alpha\textrm{-H\"ol};T}^{-\gamma}\left[
 (\tau_{0}^{1} - t)^{1-\alpha\gamma} -
 \lim_{s\rightarrow\tau_{0}^{1}}
 (\tau_{0}^{1} - s)^{1-\alpha\gamma}\right].
\end{eqnarray}
By inequalities (\ref{MAJ2MRres}) and (\ref{MAJ3MRres}) together :
\begin{equation}\label{MAJ4MRres}
\frac{a(1-\beta)}{1-\alpha\gamma}\left[
(\tau_{0}^{1} - t)^{1-\alpha\gamma} -
\lim_{s\rightarrow\tau_{0}^{1}}
(\tau_{0}^{1} - s)^{1-\alpha\gamma}\right]
\leqslant
\|\tilde w\|_{\alpha\textrm{-H\"ol};T}^{\gamma + 1}(\tau_{0}^{1} - t)^{\alpha}.
\end{equation}
If $\beta\geqslant 1/(1+\alpha) > 1-\alpha$, then $1 -\alpha\gamma\leqslant 0$ and
\begin{displaymath}
\lim_{s\rightarrow\tau_{0}^{1}}
-\frac{1}{1-\alpha\gamma}
\left(\tau_{0}^{1} - s\right)^{1-\alpha\gamma} =
\infty.
\end{displaymath}
If $1/(1+\alpha) > \beta > 1-\alpha$, inequality (\ref{MAJ4MRres}) can be rewritten as
\begin{displaymath}
\frac{a(1-\beta)}{1-\alpha\gamma}
(\tau_{0}^{1} - t)^{1-\alpha(\gamma + 1)}
\leqslant
\|\tilde w\|_{\alpha\textrm{-H\"ol};T}^{\gamma + 1},
\end{displaymath}
but $1-\alpha(\gamma + 1) < 0$ and
\begin{displaymath}
\lim_{t\rightarrow\tau_{0}^{1}}
\frac{1}{1-\alpha\gamma}
\left(\tau_{0}^{1} - t\right)^{1-\alpha(\gamma + 1)} =
\infty.
\end{displaymath}
Therefore, if $\beta > 1-\alpha$, $\tau_{0}^{1}\not\in [0,T]$.
\\
\\
An immediate consequence is that :
\begin{displaymath}
\bigcup_{\varepsilon\in ]0,y_0]}
[0,\tau_{\varepsilon}^{1}]\cap
[0,T] =
[0,T].
\end{displaymath}
Then, (\ref{EDS4}) admits a unique solution on $[0,T]$ by putting :
\begin{displaymath}
y = y^{\varepsilon}
\textrm{ on }
[0,\tau_{\varepsilon}^{1}]\cap
[0,T]
\end{displaymath}
where, $y^{\varepsilon}$ denotes the solution of (\ref{EDS4}) on $[0,\tau_{\varepsilon}^{1}]\cap[0,T]$ for every $\varepsilon\in ]0,y_0]$.
\\
\\
By Lemma \ref{EDS3change}, equation (\ref{DETmre}) admits a unique solution $\tilde\pi_V(0,x_0;w)$ on $[0,T]$, matching with $y^{\gamma + 1}e^{-b.}$.
\\
\\
Finally, since $T > 0$ is chosen arbitrarily, for $w :\mathbb R_+\rightarrow\mathbb R$ locally $\alpha$-H\"older continuous, equation (\ref{DETmre}) admits a unique solution $\tilde\pi_V(0,x_0;w)$ on $\mathbb R_+$ by putting :
\begin{displaymath}
\tilde\pi_V(0,x_0;w) =
\tilde\pi_V(0,x_0;w_{\left|[0,T]\right.})
\textrm{ on $[0,T]$}
\end{displaymath}
for every $T > 0$.
\end{proof}
\noindent
\textbf{Remarks and partial extensions :}
\begin{enumerate}
 \item Note that the statement of Lemma \ref{EDS3change} holds true when $a = 0$, and up to the time $\tau_{0}^{1}$, equation (\ref{DETmre}) has a unique explicit solution :
 \begin{displaymath}
 \forall t\in [0,\tau_{0}^{1}]
 \textrm{, }
 x_t =
 \left(x_{0}^{1-\beta} +
 \tilde w_t\right)^{\gamma + 1}
 e^{-bt}.
 \end{displaymath}
 However, in that case, $\tau_{0}^{1}$ can belong to $[0,T]$. Then, $x$ is matching with the solution of equation (\ref{DETmre}) only locally. It is sufficient for the application in pharmacokinetic provided at Section 5.
 \item For every $\alpha\in ]0,1[$, equation (\ref{EDS4}) admits a unique solution $y$ on $[0,T]$ \mbox{when :}
 \begin{equation}\label{HYPLcase}
 \inf_{s\in [0,T]}
 \tilde w_s > -y_0.
 \end{equation}
 Indeed, for every $t\in [0,\tau_{0}^{1}]$,
 \begin{displaymath}
 y_t -
 a(1-\beta)\int_{0}^{t}
 y_{s}^{-\gamma}e^{bs}ds =
 y_0 +\tilde w_t.
 \end{displaymath}
 Then,
 \begin{displaymath}
 \inf_{s\in [0,\tau_{0}^{1}]}y_s -
 a(1-\beta)\sup_{s\in [0,\tau_{0}^{1}]}\int_{0}^{s}
 y_{u}^{-\gamma}e^{bu}du
 \geqslant
 y_0 +
 \inf_{s\in [0,T]}
 \tilde w_s.
 \end{displaymath}
 Since $y$ is continuous from $[0,\tau_{0}^{1}]$ into $\mathbb R$ with $y_0 > 0$ :
 \begin{displaymath}
 \sup_{s\in [0,\tau_{0}^{1}]}
 \int_{0}^{s}
 y_{u}^{-\gamma}e^{bu}du > 0.
 \end{displaymath}
 Therefore,
 \begin{eqnarray}
  y_t & \geqslant & \inf_{s\in [0,\tau_{0}^{1}]} y_s
  \nonumber\\
  \label{Lcase}
  & \geqslant &
  y_0 +
  \inf_{s\in [0,T]}\tilde w_s > 0
 \end{eqnarray}
 by inequality (\ref{HYPLcase}). Since the right-hand side of inequality (\ref{Lcase}) is not depending on $\tau_{0}^{1}$, that hitting time is not belonging to $[0,T]$.
 \\
 \\
 By Lemma \ref{EDS3change}, equation (\ref{DETmre}) admits also a unique solution on $[0,T]$ when (\ref{HYPLcase}) is true.
 \item If $\tau_{0}^{1}\in [0,T]$, necessarily :
 \begin{displaymath}
 a(1-\beta)
 \|\tilde w\|_{\alpha\textrm{-H\"ol};T}^{-\gamma}
 \int_{t}^{\tau_{0}^{1}}
 (\tau_{0}^{1} - s)^{-\alpha\gamma}ds
 \leqslant
 \|\tilde w\|_{\alpha\textrm{-H\"ol};T}
 (\tau_{0}^{1} - t)^{\alpha}
 \end{displaymath}
 for every $t\in [0,\tau_{0}^{1}[$.
 \\
 \\
 Then, when $\beta = 1 -\alpha$, $1 -\alpha\gamma = \alpha$ and by \cite{FV08}, Theorem 6.8 :
 \begin{eqnarray*}
  a & \leqslant &
  \|\tilde w\|_{\alpha\textrm{-H\"ol};T}\\
  & \leqslant &
  C(\sigma,\alpha,b)
  \|w\|_{\alpha\textrm{-H\"ol};T}^{1/\alpha}
 \end{eqnarray*}
 with $C(\sigma,\alpha,b) = (\sigma b\alpha^2)^{1/\alpha}e^{bT}$.
 \\
 \\
 Therefore, $\tilde\pi_V(0,x_0;w)$ is defined on $[0,T]$ when $a > C(\sigma,\alpha,b)\|w\|_{\alpha\textrm{-H\"ol};T}^{1/\alpha}$.
\end{enumerate}
%

% Subsection : Upper-bound for the solution and continuity of the It map.

%
\subsection{Upper-bound for the solution and regularity of the It\^o map}
Under assumptions \ref{Eass} and \ref{HYPw}, we provide an explicit upper-bound for $\|\tilde\pi_V(0,x_0; w)\|_{\infty;T}$ and, show continuity and differentiability results for the It\^o map :
%

% Proposition : Upper-bound for the solution.

%
\begin{proposition}\label{MRint}
Under assumptions \ref{Eass} and \ref{HYPw}, for $a > 0$ and $b\geqslant 0$, with any initial condition $x_0 > 0$,
\begin{eqnarray*}
 \|\tilde\pi_V(0,x_0;w)\|_{\infty;T}
 & \leqslant &
 \left[
 x_{0}^{1-\beta} + a(1-\beta)e^{bT}x_{0}^{-\beta}T +\right.\\
 & &
 \left.
 \sigma(b\vee 2)(1-\beta)(1 + T)e^{b(1-\beta)T}\|w\|_{\infty;T}
 \right]^{\gamma + 1}.
\end{eqnarray*}
\end{proposition}
%

% Proof.

%
\begin{proof}
Consider $y_0 = x_{0}^{1-\beta}$, $y$ the solution of (\ref{EDS4}) with initial condition $y_0$ and
\begin{displaymath}
\tau_{y_0}^{2} =
\sup\left\{
t\in [0,T] :
y_t\leqslant y_0\right\}.
\end{displaymath}
On one hand, we consider the two following cases :
\begin{enumerate}
 \item If $t < \tau_{y_0}^{2}$ :
 \begin{displaymath}
 y_{\tau_{y_0}^{2}} - y_t =
 a(1-\beta)\int_{t}^{\tau_{y_0}^{2}}
 y_{s}^{-\gamma}e^{bs}ds +
 \tilde w_{\tau_{y_0}^{2}} -
 \tilde w_t. 
 \end{displaymath}
 Then, by definition of $\tau_{y_0}^{2}$ :
 \begin{equation}\label{EQMRint}
 y_t +
 a(1-\beta)\int_{t}^{\tau_{y_0}^{2}}
 y_{s}^{-\gamma}e^{bs}ds =
 y_0 +
 \tilde w_t -
 \tilde w_{\tau_{y_0}^{2}}.
 \end{equation}
 Therefore, since each term of the sum in the left-hand side of equality (\ref{EQMRint}) are positive from Proposition \ref{MRres} :
 \begin{displaymath}
 0 < y_t\leqslant
 y_0 +
 |\tilde w_t -
 \tilde w_{\tau_{y_0}^{2}}|.
 \end{displaymath}
 \item If $t\geqslant\tau_{y_0}^{2}$ ; by definition of $\tau_{y_0}^{2}$, $y_t\geqslant y_0$ and then, $y_{t}^{-\gamma}\leqslant y_{0}^{-\gamma}$. Therefore,
 \begin{displaymath}
 y_0\leqslant y_t
 \leqslant
 y_0 + a(1-\beta)e^{bT}y_{0}^{-\gamma}T + |\tilde w_t - \tilde w_{\tau_{y_0}^{2}}|.
 \end{displaymath}
\end{enumerate}
On the other hand, by using the integration by parts formula, for every $t\in [0,T]$,
\begin{eqnarray*}
 |\tilde w_t - \tilde w_{\tau_{y_0}^{2}}|
 & = &
 \sigma(1-\beta)\left|\int_{\tau_{y_0}^{2}}^{t}
 e^{b(1-\beta)s}dw_s\right|\\
 & = &
 \sigma(1-\beta)\left|
 e^{b(1-\beta)t}w_t -
 e^{b(1-\beta)\tau_{y_0}^{2}}w_{\tau_{y_0}^{2}} -
 b(1-\beta)\int_{\tau_{y_0}^{2}}^{t}
 e^{b(1-\beta)s}w_sds
 \right|\\
 & \leqslant &
 \sigma(1-\beta)\left[
 2 + b(1-\beta)T\right]
 e^{b(1-\beta)T}\|w\|_{\infty;T}\\
 & \leqslant &
 \sigma(b\vee 2)(1-\beta)(1 + T)e^{b(1-\beta)T}\|w\|_{\infty;T},
\end{eqnarray*}
because $(1-\beta)^2\leqslant 1-\beta\leqslant 1$.
\\
\\
Therefore, by putting cases 1 and 2 together ; for every $t\in [0,T]$,
\begin{equation}\label{MAJMRint}
0 < y_t
\leqslant
y_0 + a(1-\beta)e^{bT}y_{0}^{-\gamma}T +
\sigma(b\vee 2)(1-\beta)(1 + T)e^{b(1-\beta)T}\|w\|_{\infty;T}.
\end{equation}
That achieves the proof because, $\tilde\pi_V(0,x_0;w) = y^{\gamma + 1}e^{-b.}$ and the right hand side of inequality (\ref{MAJMRint}) is not depending on $t$.
\end{proof}
\noindent
\textbf{Remark.} In particular, by Proposition \ref{MRint}, $\|\tilde\pi_V(0,x_0;w)\|_{\infty;T}$ does not explode when $a\rightarrow 0$ or/and $b\rightarrow 0$.
\\
\\
\textbf{Notation.} In the sequel, for every $R > 0$,
\begin{displaymath}
B_{\alpha}(0,R) :=
\left\{w\in C^{\alpha\textrm{-H\"ol}}([0,T];\mathbb R) :
\left\|w\right\|_{\alpha\textrm{-H\"ol};T}\leqslant R\right\}.
\end{displaymath}
%

% Proposition : Continuity of the It map.

%
\begin{proposition}\label{MRcont}
Under Assumption \ref{Eass}, for $a > 0$ and $b\geqslant 0$, $\tilde\pi_V(0,.)$ is a continuous map from $\mathbb R_{+}^{*}\times C^{\alpha\textrm{-H\"ol}}([0,T];\mathbb R)$ into $C^0([0,T];\mathbb R)$. Moreover, $\tilde\pi_V(0,.)$ is Lipschitz continuous from $[r,R_1]\times B_{\alpha}(0,R_2)$ into $C^0([0,T];\mathbb R)$ for every $R_1 > r > 0$ and $R_2 > 0$.
\end{proposition}
%

% Proof.

%
\begin{proof}
Consider $(x_{0}^{1},w^1)$ and $(x_{0}^{2},w^2)$ belonging to $\mathbb R_{+}^{*}\times C^{\alpha\textrm{-H\"ol}}([0,T];\mathbb R)$.
\\
\\
For $i = 1,2$, we put $y_{0}^{i} = (x_{0}^{i})^{1-\beta}$ and $y^i = I(y_{0}^{i},\tilde w^i)$ where,
\begin{displaymath}
\forall t\in [0,T]
\textrm{, }
\tilde w_{t}^{i} =
\int_{0}^{t}\vartheta_sdw_{s}^{i}
\end{displaymath}
and, with notations of equation (\ref{EDS4}), $I$ is the map defined by :
\begin{displaymath}
I(y_0,\tilde w) = y_0 + a(1-\beta)\int_{0}^{.}I_{s}^{-\gamma}(y_0,\tilde w)e^{bs}ds + \tilde w.
\end{displaymath}
We also put :
\begin{displaymath}
\tau^3 =
\inf\left\{s\in [0,T] : y_{s}^{1} = y_{s}^{2}\right\}.
\end{displaymath}
On one hand, we consider the two following cases :
\begin{enumerate}
 \item Consider $t\in [0,\tau^3]$ and suppose that $y_{0}^{1}\geqslant y_{0}^{2}$.
 \\
 \\
 Since $y^1$ and $y^2$ are continuous on $[0,T]$ by construction, for every $s\in [0,\tau^3]$, $y_{s}^{1}\geqslant y_{s}^{2}$ and then,
 \begin{displaymath}
 \left(y_{s}^{1}\right)^{-\gamma} -
 \left(y_{s}^{2}\right)^{-\gamma}\leqslant 0.
 \end{displaymath}
 Therefore,
 \begin{eqnarray*}
  \left|y_{t}^{1} - y_{t}^{2}\right| & = &
  y_{t}^{1} - y_{t}^{2}\\
  & = &
  y_{0}^{1} - y_{0}^{2} +
  a(1-\beta)\int_{0}^{t}e^{bs}[
  \left(y_{s}^{1}\right)^{-\gamma} -
  \left(y_{s}^{2}\right)^{-\gamma}]ds +
  \tilde w_{t}^{1} - \tilde w_{t}^{2}\\
  & \leqslant &
  |y_{0}^{1} - y_{0}^{2}| +
  \|\tilde w^1 - \tilde w^2\|_{\infty;T}.
 \end{eqnarray*}
 Symmetrically, one can show that this inequality is still true when $y_{0}^{1}\leqslant y_{0}^{2}$.
 \item Consider $t\in [\tau_3,T]$,
 \begin{displaymath}
 \tau^3(t) =
 \sup\left\{s\in\left[\tau^3,t\right] : y_{s}^{1} = y_{s}^{2}\right\}
 \end{displaymath}
 and suppose that $y_{t}^{1}\geqslant y_{t}^{2}$.
 \\
 \\
 Since $y^1$ and $y^2$ are continuous on $[0,T]$ by construction, for every $s\in [\tau^3(t),t]$, $y_{s}^{1}\geqslant y_{s}^{2}$ and then,
 \begin{displaymath}
 \left(y_{s}^{1}\right)^{-\gamma} -
 \left(y_{s}^{2}\right)^{-\gamma}\leqslant 0.
 \end{displaymath}
 Therefore,
 \begin{eqnarray*}
  \left|y_{t}^{1} - y_{t}^{2}\right| & = &
  y_{t}^{1} - y_{t}^{2}\\
  & = &
  a(1-\beta)
  \int_{\tau^3(t)}^{t}e^{bs}[
  \left(y_{s}^{1}\right)^{-\gamma} -
  \left(y_{s}^{2}\right)^{-\gamma}]ds +
  \tilde w_{t}^{1} - \tilde w_{t}^{2} -
  [\tilde w_{\tau^3(t)}^{1} - \tilde w_{\tau^3(t)}^{2}]\\
  & \leqslant &
  2\|\tilde w^1 - \tilde w^2\|_{\infty;T}.
 \end{eqnarray*}
 Symmetrically, one can show that this inequality is still true when $y_{t}^{1}\leqslant y_{t}^{2}$.
\end{enumerate}
By putting these cases together and since the obtained upper-bounds are not depending on $t$ :
\begin{equation}\label{yLip}
\|y^1 - y^2\|_{\infty;T}
\leqslant
|y_{0}^{1} - y_{0}^{2}| +
2T^{\alpha}\|\tilde w^1 - \tilde w^2\|_{\alpha\textrm{-H\"ol};T}. 
\end{equation}
Then, $I$ is continuous from $\mathbb R_{+}^{*}\times C^{\alpha\textrm{-H\"ol}}([0,T];\mathbb R)$ into $C^0([0,T];\mathbb R)$.
\\
\\
For any $\alpha$-H\"older continuous function $w : [0,T]\rightarrow\mathbb R$, from Lemma \ref{EDS3change} and Proposition \ref{MRres} :
\begin{displaymath}
\tilde\pi_V(0,x_0;w) =
e^{-b.}I^{\gamma + 1}\left[x_{0}^{1-\beta},\mathcal Y(\vartheta,w)\right].
\end{displaymath}
Moreover, by \cite{FV08}, Proposition 6.12, $\mathcal Y(\vartheta,.)$ is continuous from $C^{\alpha\textrm{-H\"ol}}([0,T];\mathbb R)$ into itself. Therefore, $\tilde\pi_V(0,.)$ is continuous from $\mathbb R_{+}^{*}\times C^{\alpha\textrm{-H\"ol}}([0,T];\mathbb R)$ into $C^0([0,T];\mathbb R)$ by composition.
\\
\\
On the other hand, consider $R_1 > r > 0$ and $R_2 > 0$. By Proposition \ref{MRint}, there exists $C > 0$ such that :
\begin{displaymath}
\forall (x_0,w)\in [r,R_1]\times B_{\alpha}(0,R_2)
\textrm{, }
\|I[x_{0}^{1-\beta},\mathcal Y(\vartheta,w)]\|_{\infty;T}
\leqslant
C(r^{-\gamma} + R_1 + R_2).
\end{displaymath}
Then, for every $(x_{0}^{1},w^1),(x_{0}^{2},w^2)\in [r,R_1]\times B_{\alpha}(0,R_2)$,
\begin{eqnarray*}
 \|\tilde\pi_V(0,x_0;w^1) - \tilde\pi_V(0,x_0;w^2)\|_{\infty;T}
 & \leqslant &
 (\gamma + 1)C^{\gamma}(r^{-\gamma} + R_1 + R_2)^{\gamma}\times\\
 & &
 \left[(1-\beta)r^{-\beta}|x_{0}^{1} - x_{0}^{2}|\right. +\\
 & &
 \left.
 2T^{\alpha}\|\mathcal Y(\vartheta,w^1) -\mathcal Y(\vartheta,w^2)\|_{\alpha\textrm{-H\"ol};T}\right]
\end{eqnarray*}
by inequality (\ref{yLip}). Since $\mathcal Y(\vartheta,.)$ is Lipschitz continuous from bounded sets of $C^{\alpha\textrm{-H\"ol}}([0,T];\mathbb R)$ into $C^{\alpha\textrm{-H\"ol}}([0,T];\mathbb R)$ (cf. \cite{FV08}, Proposition 6.11), that achieves the proof.
\end{proof}
\noindent
In order to study the regularity of the solution of equation (\ref{DETmre}) with respect to parameters $a,b\geqslant 0$ characterizing the vector field $V$, let's denote by $x(a,b)$ (resp. $y(a,b)$) the solution of equation (\ref{DETmre}) (resp. (\ref{EDS4})) up to $\tau_{0}^{1}\wedge T$.
%

% Proposition : Monotonicity of the It map with respect to the vector field.

%
\begin{proposition}\label{ITOmono}
Under assumptions \ref{Eass} and \ref{HYPw}, for every $a,b\geqslant 0$, $x(0,b)\leqslant x(a,b)\leqslant x(a,0)$.
\end{proposition}
%

% Proof.

%
\begin{proof}
On one hand, consider $a\geqslant 0$, $b > 0$ and $t\in [0,\tau_{0}^{1}\wedge T]$ :
\begin{displaymath}
y_t(a,b) -y_t(0,b) =
a(1-\beta)\int_{0}^{t}
y_{s}^{-\gamma}(a,b)e^{bs}ds
\geqslant 0,
\end{displaymath}
because $y_s(a,b)\geqslant 0$ for every $s\in [0,T]$ by Proposition \ref{MRres}.
\\
\\
Then, by Lemma \ref{EDS3change} :
\begin{displaymath}
x(0,b)\leqslant x(a,b).
\end{displaymath}
On the other hand, consider $a > 0$, $b\geqslant 0$, $z(a,b) = x^{1-\beta}(a,b)$ and $t_1,t_2\in [0,\tau_{0}^{1}\wedge T]$ such that : $t_1 < t_2$, $x_{t_1}(a,0) = x_{t_1}(a,b)$ and $x_s(a,0) < x_s(a,b)$ for every $s\in [t_1,t_2]$. As at Lemma \ref{EDS3change}, by the change of variable formula (Theorem \ref{CVformula}), for every $t\in [t_1,t_2]$,
\begin{eqnarray*}
 z_t(a,0) - z_t(a,b) & = &
 z_t(a,0) - z_{t_1}(a,0) -
 [z_t(a,b) - z_{t_1}(a,b)]\\
 & = &
 a(1 -\beta)
 \int_{t_1}^{t}
 [x_{s}^{-\beta}(a,0) - x_{s}^{-\beta}(a,b)]ds +\\
 & &
 b(1-\beta)
 \int_{t_1}^{t}x_{s}^{-\beta}(a,b)ds\\
 & \geqslant &
 a(1 -\beta)
 \int_{t_1}^{t}
 [x_{s}^{-\beta}(a,0) - x_{s}^{-\beta}(a,b)]ds,
\end{eqnarray*}
because $x_s(a,b)\geqslant 0$ for every $s\in [t_1,t]$ by Proposition \ref{MRres}.
\\
\\
Since $x_s(a,0) < x_s(a,b)$ for every $s\in [t_1,t_2]$ by assumption, necessarily :
\begin{displaymath}
z_t(a,0) - z_t(a,b) < 0
\end{displaymath}
and
\begin{displaymath}
\int_{t_1}^{t}\left[
x_{s}^{-\beta}(a,0) -x_{s}^{-\beta}(a,b)
\right]ds
\geqslant 0.
\end{displaymath}
Therefore, it's impossible, and for every $t\in [0,\tau_{0}^{1}\wedge T]$, $x_t(a,0)\geqslant x_t(a,b)$.
\end{proof}
%

% Proposition : Continuity of the It map with respect to the vector field.

%
\begin{proposition}\label{ITOVFcont}
Under assumptions \ref{Eass} and \ref{HYPw}, $(a,b)\mapsto x(a,b)$ is a continuous map from $(\mathbb R_{+}^{*})^2$ into $C^0([0,T];\mathbb R)$.
\end{proposition}
%

% Proof.

%
\begin{proof}
Consider $a^0,a,b^0,b > 0$ and $\tilde w^0,\tilde w : [0,T]\rightarrow\mathbb R$ two functions defined by :
\begin{displaymath}
\forall t\in [0,T]\textrm{, }
\tilde w_{t}^{0} =
\sigma(1-\beta)\int_{0}^{t}
e^{b^0(1-\beta)s}dw_s
\textrm{ and }
\tilde w_t =
\sigma(1-\beta)\int_{0}^{t}
e^{b(1-\beta)s}dw_s.
\end{displaymath}
For every $t\in [0,T]$,
\begin{eqnarray*}
 y_t(a,b) - y_t(a^0,b^0)
 & = &
 a(1-\beta)
 \int_{0}^{t}
 y_{s}^{-\gamma}(a,b)e^{bs}ds -\\
 & &
 a^0(1-\beta)
 \int_{0}^{t}
 y_{s}^{-\gamma}(a^0,b^0)e^{b^0s}ds +
 \tilde w_t -
 \tilde w_{t}^{0}\\
 & = &
 a(1-\beta)
 \int_{0}^{t}\left[y_{s}^{-\gamma}(a,b) - y_{s}^{-\gamma}(a^0,b^0)\right]e^{bs}ds +\\
 & &
 (1-\beta)\int_{0}^{t}
 (ae^{bs} - a^0e^{b^0s})y_{s}^{-\gamma}(a^0,b^0)ds +
 \tilde w_t - \tilde w_{t}^{0}.
\end{eqnarray*}
As at Proposition \ref{MRcont}, by using the monotonicity of $u\in\mathbb R_{+}^{*}\mapsto u^{-\gamma}$ together with appropriate crossing times :
\begin{eqnarray*}
 \|y(a,b) - y(a^0,b^0)\|_{\infty;T}
 & \leqslant &
 (1-\beta)T\|ae^{b.} - a^0e^{b^0.}\|_{\infty;T}
 \|y^{-\gamma}(a^0,b^0)\|_{\infty;T} +\\
 & &
 2T^{\alpha}
 \|\tilde w - \tilde w^0\|_{\alpha\textrm{-H\"ol};T}\\
 & \leqslant &
 (1-\beta)T\left[
 |a - a_0|e^{bT} +
 a_0e^{(b\vee b_0)T}T|b - b_0|\right]\times\\
 & &
 \|y^{-\gamma}(a^0,b^0)\|_{\infty;T} +
 2T^{\alpha}
 \|\tilde w - \tilde w^0\|_{\alpha\textrm{-H\"ol};T}.
\end{eqnarray*}
Moreover, by \cite{FV08}, Theorem 6.8 :
\begin{eqnarray*}
 \|\tilde w - \tilde w^0\|_{\alpha\textrm{-H\"ol};T}
 & \leqslant &
 \sigma(1-\beta)
 \|w\|_{\alpha\textrm{-H\"ol};T}\|e^{b^0(1-\beta).} - e^{b(1-\beta).}\|_{1\textrm{-H\"ol};T}\\
 & \leqslant &
 \sigma(1-\beta)^2
 |b - b^0|\times\\
 & &
 \|w\|_{\alpha\textrm{-H\"ol};T}
 \left[e^{b^0(1-\beta)T} + b(1-\beta)e^{(b\vee b^0)(1-\beta)T}T\right].
\end{eqnarray*}
These inequalities imply that :
\begin{displaymath}
\lim_{(a,b)\rightarrow (a^0,b^0)}
\left\|y(a,b) - y(a^0,b^0)\right\|_{\infty;T} = 0.
\end{displaymath}
Therefore, $(a,b)\mapsto x(a,b) = e^{-b.}y^{\gamma + 1}(a,b)$ is a continuous map from $(\mathbb R_{+}^{*})^2$ into $C^0([0,T];\mathbb R)$.
\end{proof}
\noindent
Let's now show the continuous differentiability of the It\^o map with respect to the initial condition and the driving signal :
%

% Proposition : Differentiability of the It map with respect to the initial condition and the driving signal.

%
\begin{proposition}\label{ITOdiff}
Under Assumption \ref{Eass}, for $a > 0$ and $b\geqslant 0$, $\tilde\pi_V(0,.)$ is continuously differentiable from $\mathbb R_{+}^{*}\times C^{\alpha\textrm{-H\"ol}}([0,T];\mathbb R)$ into $C^0([0,T];\mathbb R)$.
\end{proposition}
%

% Proof.

%
\begin{proof}
In a sake of readability, the space $\mathbb R_{+}^{*}\times C^{\alpha\textrm{-H\"ol}}([0,T];\mathbb R)$ is denoted by $E$.
\\
\\
Consider $(x_{0}^{0},w^0)\in E$, $x^0 := \tilde\pi_V(0,x_{0}^{0};w^0)$,
\begin{displaymath}
m_0\in
\left]0,\min_{t\in [0,T]}x_{t}^{0}\right[
\textrm{ and }
\varepsilon_0 :=
-m_0 + \min_{t\in [0,T]}x_{t}^{0}.
\end{displaymath}
Since $\tilde\pi_V(0,.)$ is continuous from $E$ into $C^0([0,T];\mathbb R)$ by Proposition \ref{MRcont} :
\begin{eqnarray}
 \forall\varepsilon\in ]0,\varepsilon_0]
 \textrm{, }
 \exists\eta > 0 & : &
 \forall (x_0,w)\in E,
 \nonumber\\
 \label{continuity_z}
 (x_0,w) & \in & B_E((x_{0}^{0},w^0);\eta)
 \Longrightarrow
 \|\tilde\pi_V(0,x_0;w) - x^0\|_{\infty;T} < \varepsilon\leqslant\varepsilon_0.
\end{eqnarray}
In particular, for every $(x_0,w)\in B_E((x_{0}^{0},w^0);\eta)$, the function $\tilde\pi_V(0,x_0;w)$ is $[m_0,M_0]$-valued with $[m_0,M_0]\subset\mathbb R_{+}^{*}$ and
\begin{displaymath}
M_0 :=
-m_0 +
\min_{t\in [0,T]}
x_{t}^{0} +
\max_{t\in [0,T]}
x_{t}^{0}.
\end{displaymath}
In \cite{FV08}, the continuous differentiability of the It\^o map with respect to the initial condition and the driving signal is established at theorems 11.3 and 11.6. In order to derive the It\^o map with respect to the driving signal at point $w^0$ in the direction $h\in C^{\kappa\textrm{-H\"ol}}([0,T];\mathbb R^d)$, $\kappa\in ]0,1[$ has to satisfy the condition $\alpha +\kappa > 1$ to ensure the existence of the geometric $1/\alpha$-rough path over $w^0 + \varepsilon h$ ($\varepsilon > 0$) provided at \cite{FV08}, Theorem 9.34 when $d > 1$. When $d = 1$, that condition can be dropped by (\ref{GRP1dim}). Therefore, since the vector field $V$ is $C^{\infty}$ on $[m_0,M_0]$, $\tilde\pi_V(0,.)$ is continuously differentiable from $B_E((x_{0}^{0},w^0);\eta)$ into $C^0([0,T];\mathbb R)$.
\\
\\
In conclusion, since $(x_{0}^{0},w^0)$ has been arbitrarily chosen, $\tilde\pi_V(0,.)$ is continuously differentiable from $\mathbb R_{+}^{*}\times C^{\alpha\textrm{-H\"ol}}([0,T];\mathbb R)$ into $C^0([0,T];\mathbb R)$.
\end{proof}
%

% Subsection : A converging approximation.

%
\subsection{A converging approximation}
In order to provide a converging approximation for equation (\ref{DETmre}), we first prove the convergence of the implicit Euler approximation $(y^n,n\in\mathbb N^*)$ for equation (\ref{EDS4}) :
\begin{equation}\label{EULdef}
\left\{
\begin{array}{rcl}
y_{0}^{n} & = & y_0 > 0\\
y_{k + 1}^{n} & = &
\displaystyle{y_{k}^{n} +\frac{a(1-\beta)T}{n}(y_{k + 1}^{n})^{-\gamma}e^{bt_{k + 1}^{n}} +
\tilde w_{t_{k + 1}^{n}} - \tilde w_{t_{k}^{n}}}
\end{array}
\right.
\end{equation}
where, for $n\in\mathbb N^*$, $t_{k}^{n} = kT/n$ and $k\leqslant n$ while $y_{k + 1}^{n} > 0$.
\\
\\
\textbf{Remark.} On the implicit Euler approximation in stochastic analysis, cf. F. Malrieu \cite{FM00} and, F. Malrieu and D. Talay \cite{MT06} for example.
\\
\\
The following proposition shows that the implicit step-$n$ Euler approximation $y^n$ is defined on $\{1,\dots,n\}$ :
%

% Proposition : Existence of the Euler approximation.

%
\begin{proposition}\label{EULres}
Under Assumption \ref{HYPw}, for $a > 0$ and $b\geqslant 0$, equation (\ref{EULdef}) admits a unique solution $(y^n,n\in\mathbb N^*)$. Moreover,
\begin{displaymath}
\forall n\in\mathbb N^*\textrm{, }
\forall k = 0,\dots, n\textrm{, }
y_{k}^{n} > 0.
\end{displaymath}
\end{proposition}
%

% Proof.

%
\begin{proof}
Let $f$ be the function defined on $\mathbb R_{+}^{*}\times\mathbb R\times\mathbb R_{+}^{*}$ by :
\begin{displaymath}
\forall A\in\mathbb R\textrm{, }
\forall x,B > 0\textrm{, }
f\left(x,A,B\right) =
x - Bx^{-\gamma} - A.
\end{displaymath}
On one hand, for every $A\in\mathbb R$ and $B > 0$, $f(.,A,B)\in C^{\infty}(\mathbb R_{+}^{*};\mathbb R)$ and for every $x > 0$,
\begin{displaymath}
\partial_x f\left(x,A,B\right) = 1 + B\gamma x^{-(\gamma + 1)} > 0.
\end{displaymath}
Then, $f(.,A,B)$ increase on $\mathbb R_{+}^{*}$. Moreover,
\begin{displaymath}
\lim_{x\rightarrow 0^+}
f\left(x,A,B\right) = -\infty
\textrm{ and }
\lim_{x\rightarrow\infty}
f\left(x,A,B\right) = \infty.
\end{displaymath}
Therefore, since $f$ is continuous on $\mathbb R_{+}^{*}\times\mathbb R\times\mathbb R_{+}^{*}$ :
\begin{equation}\label{EULres1}
\forall A\in\mathbb R\textrm{, }
\forall B > 0\textrm{, }
\exists ! x > 0 :
f\left(x,A,B\right) = 0.
\end{equation}
On the other hand, for every $n\in\mathbb N^*$, equation (\ref{EULdef}) can be rewritten as follow :
\begin{equation}\label{EULres2}
f\left[y_{k + 1}^{n},
y_{k}^{n} + \tilde w_{t_{k + 1}^{n}} -
\tilde w_{t_{k}^{n}},
\frac{a(1-\beta)T}{n}e^{bt_{k + 1}^{n}}\right] = 0.
\end{equation}
In conclusion, by recurrence, equation (\ref{EULres2}) admits a unique strictly positive solution $y_{k + 1}^{n}$.
\\
\\
Necessarily, $y_{k}^{n} > 0$ for $k = 0,\dots, n$.
\\
\\
That achieves the proof.
\end{proof}
\noindent
For every $n\in\mathbb N^*$, consider the function $y^n : [0,T]\rightarrow\mathbb R_{+}^{*}$ such that :
\begin{displaymath}
y_{t}^{n} =
\sum_{k = 0}^{n - 1}
\left[
y_{k}^{n} +
\frac{y_{k + 1}^{n} - y_{k}^{n}}{t_{k + 1}^{n} - t_{k}^{n}}
(t - t_{k}^{n})\right]
\mathbf 1_{[t_{k}^{n},t_{k + 1}^{n}[}(t)
\end{displaymath}
for every $t\in [0,T]$.
\\
\\
The following lemma provides an explicit upper-bound for $(n,t)\in\mathbb N^*\times [0,T]\mapsto y_{t}^{n}$. It is crucial in order to prove probabilistic convergence results at Section 4.
%

% Lemma : Upper-bound for the Euler approximation.

%
\begin{lemma}\label{EULcv}
Under Assumption \ref{HYPw}, for $a > 0$ and $b\geqslant 0$ :
\begin{eqnarray*}
 \sup_{n\in\mathbb N^*}
 \left\|y^n\right\|_{\infty;T}
 & \leqslant &
 y_0 + a(1-\beta)e^{bT}y_{0}^{-\gamma}T +\\
 & &
 \sigma(b\vee 2)(1-\beta)(1 + T)e^{b(1-\beta)T}\|w\|_{\infty;T}.
\end{eqnarray*}
\end{lemma}
%

% Proof.

%
\begin{proof}
Similar to the proof of Proposition \ref{MRint}.
\\
\\
First of all, by applying (\ref{EULdef}) recursively between integers $0\leqslant l < k\leqslant n$ and a change of variable :
\begin{equation}\label{EULRdef}
y_{k}^{n} - y_{l}^{n}=
\frac{a(1-\beta)T}{n}
\sum_{i = l + 1}^{k}
\left(y_{i}^{n}\right)^{-\gamma}e^{bt_{i}^{n}} +
\tilde w_{t_{k}^{n}} -
\tilde w_{t_{l}^{n}}.
\end{equation}
Consider $n\in\mathbb N^*$ and
\begin{displaymath}
k_{y_0} = \max\left\{k = 0,\dots,n : y_{k}^{n}\leqslant y_0\right\}.
\end{displaymath}
For each $k = 1,\dots,n$, we consider the two following cases :
\begin{enumerate}
 \item If $k < k_{y_0}$, from equality (\ref{EULRdef}) :
 \begin{displaymath}
 y_{k_{y_0}}^{n} - y_{k}^{n} =
 \frac{a(1-\beta)T}{n}
 \sum_{i = k + 1}^{k_{y_0}}
 \left(y_{i}^{n}\right)^{-\gamma}e^{bt_{i}^{n}} +
 \tilde w_{t_{k_{y_0}}^{n}} -
 \tilde w_{t_{k}^{n}}.
 \end{displaymath}
 Then,
 \begin{equation}\label{EQEULcv}
 y_{k}^{n} +
 \frac{a(1-\beta)T}{n}
 \sum_{i = k + 1}^{k_{y_0}}
 \left(y_{i}^{n}\right)^{-\gamma}e^{bt_{i}^{n}} =
 y_{k_{y_0}}^{n} +
 \tilde w_{t_{k}^{n}} -
 \tilde w_{t_{k_{y_0}}^{n}}.
 \end{equation}
 Therefore, since each term of the sum in the left-hand side of equality (\ref{EQEULcv}) are positive from Proposition \ref{EULres} :
 \begin{eqnarray*}
  0 < y_{k}^{n} & \leqslant &
  y_{k}^{n} +
  \frac{a(1-\beta)T}{n}
  \sum_{i = k + 1}^{k_{y_0}}
  \left(y_{i}^{n}\right)^{-\gamma}e^{bt_{i}^{n}}\\
  & \leqslant &
  y_0 + |\tilde w_{t_{k}^{n}} -
  \tilde w_{t_{k_{y_0}}^{n}}|
 \end{eqnarray*}
 because $y_{k_{y_0}}^{n}\leqslant y_0$.
 \item If $k > k_{y_0}$ ; by definition of $k_{y_0}$, for $i = k_{y_0} + 1,\dots, k$, $y_{i}^{n} > y_0$ and then, $(y_{i}^{n})^{-\gamma}\leqslant y_{0}^{-\gamma}$. Therefore, from equality (\ref{EULRdef}) :
 \begin{eqnarray*}
  y_0\leqslant y_{k}^{n} & = &
  y_{k_{y_0}}^{n} +
  \frac{a(1-\beta)T}{n}
  \sum_{i = k_{y_0} + 1}^{k}
  \left(y_{i}^{n}\right)^{-\gamma}e^{bt_{i}^{n}} +
  \tilde w_{t_{k}^{n}} -
  \tilde w_{t_{k_{y_0}}^{n}}\\
  & \leqslant &
  y_0 + a(1-\beta)e^{bT}y_{0}^{-\gamma}T +
  |\tilde w_{t_{k}^{n}} -
  \tilde w_{t_{k_{y_0}}^{n}}|.
 \end{eqnarray*}
\end{enumerate}
As at Proposition \ref{MRint} :
\begin{eqnarray}
 \sup_{t\in [0,T]}y_{t}^{n}
 & \leqslant &
 \max_{k = 0,\dots,n} y_{k}^{n}
 \nonumber\\
 & \leqslant &
 \label{MAJEULcv}
 y_0 + a(1-\beta)e^{bT}y_{0}^{-\gamma}T +
 \sigma(b\vee 2)(1-\beta)(1 + T)e^{b(1-\beta)T}\|w\|_{\infty;T}.
\end{eqnarray}
That achieves the proof because the right hand side of inequality (\ref{MAJEULcv}) is not depending on $n$.
\end{proof}
\noindent
With ideas of A. Lejay \cite{LEJ10}, Proposition 5, we show that $(y^n,n\in\mathbb N^*)$ converges and provide a rate of convergence :
%

% Theorem : Limit Euler approximation.

%
\begin{theorem}\label{EULlim}
Under assumptions \ref{Eass} and \ref{HYPw}, for $a > 0$ and $b\geqslant 0$ ; $(y^n,n\in\mathbb N^*)$ is uniformly converging on $[0,T]$ to $y$, the solution of equation (\ref{EDS4}) with initial condition $y_0$, with rate $n^{-\alpha\min(1,\gamma)}$.
\end{theorem}
%

% Proof.

%
\begin{proof}
It follows the same pattern that Proof of \cite{LEJ10}, Proposition 5.
\\
\\
Consider $n\in\mathbb N^*$, $t\in [0,T]$ and $y$ the solution of equation (\ref{EDS4}) with initial condition $y_0 > 0$. Since $(t_{k}^{n};k = 0,\dots,n)$ is a subdivision of $[0,T]$, there exists an integer $0\leqslant k\leqslant n - 1$ such that $t\in [t_{k}^{n},t_{k + 1}^{n}[$.
\\
\\
First of all, note that :
\begin{equation}\label{EULmaj1}
|y_{t}^{n} - y_t|\leqslant
|y_{t}^{n} - y_{k}^{n}| +
|y_{k}^{n} - z_{k}^{n}| + 
|z_{k}^{n} - y_t|
\end{equation}
where, $z_{i}^{n} = y_{t_{i}^{n}}$ for $i = 0,\dots,n$. Since $y$ is the solution of equation (\ref{EDS4}), $z_{k}^{n}$ and $z_{k + 1}^{n}$ satisfy :
\begin{displaymath}
z_{k + 1}^{n} =
z_{k}^{n} +
\frac{a(1-\beta)T}{n}(z_{k + 1}^{n})^{-\gamma}e^{bt_{k + 1}^{n}} +
\tilde w_{t_{k + 1}^{n}} -
\tilde w_{t_{k}^{n}} +
\varepsilon_{k}^{n}
\end{displaymath}
where,
\begin{displaymath}
\varepsilon_{k}^{n} =
a(1-\beta)\int_{t_{k}^{n}}^{t_{k + 1}^{n}}
(y_{s}^{-\gamma}e^{bs} -
y_{t_{k + 1}^{n}}^{-\gamma}e^{bt_{k + 1}^{n}})ds.
\end{displaymath}
In order to conclude, we have to show that $|y_{k}^{n} - z_{k}^{n}|$ is bounded by a quantity not depending on $k$ and converging to $0$ when $n$ goes to infinity :
\\
\\
On one hand, for every $(u,v)\in\Delta_T$,
\begin{eqnarray*}
 \left|e^{bv}y_{v}^{-\gamma} - e^{bu}y_{u}^{-\gamma}\right|
 & = &
 \left|\frac{e^{bv}y_{u}^{\gamma} - e^{bu}y_{v}^{\gamma}}{y_{v}^{\gamma}y_{u}^{\gamma}}\right|\\
 & \leqslant &
 \frac{1}{|y_uy_v|^{\gamma}}
 \left(e^{bv}|y_{u}^{\gamma} - y_{v}^{\gamma}| +
 |y_v|^{\gamma}|e^{bu} - e^{bv}|\right)\\
 & \leqslant &
 e^{bT}\|y^{-\gamma}\|_{\infty;T}^{2}
 \left(\|y\|_{\alpha\textrm{-H\"ol};T}^{\min(1,\gamma)}|v - u|^{\alpha\min(1,\gamma)} +
 b\|y\|_{\infty;T}^{\gamma}|v - u|\right)
\end{eqnarray*}
because $s\in\mathbb R_+\mapsto s^{\gamma}$ is $\gamma$-H\"older continuous with constant $1$ if $\gamma\in ]0,1]$ and locally Lipschitz continuous otherwise, $y$ is $\alpha$-H\"older continuous and admits a strictly positive minimum on $[0,T]$, and $s\in [0,T]\mapsto e^{bs}$ is Lipschitz continuous with constant $be^{bT}$. In particular, if $|v - u|\leqslant 1$,
\begin{displaymath}
|e^{bv}y_{v}^{-\gamma} - e^{bu}y_{u}^{-\gamma}|
\leqslant
e^{bT}\|y^{-\gamma}\|_{\infty;T}^{2}
\left(\|y\|_{\alpha\textrm{-H\"ol};T}^{\mu} +
b\|y\|_{\infty;T}^{\gamma}\right)
|v - u|^{\alpha\mu}
\end{displaymath}
where $\mu = \min(1,\gamma)$.
\\
\\
Then, for $i = 0,\dots,k$,
\begin{eqnarray}
 |\varepsilon_{i}^{n}|
 & \leqslant &
 a(1 -\beta)
 \int_{t_{i}^{n}}^{t_{i + 1}^{n}}
 |y_{s}^{-\gamma}e^{bs} -
 y_{t_{i + 1}^{n}}^{-\gamma}e^{bt_{i + 1}^{n}}|ds
 \nonumber\\
 & \leqslant &
 a(1 -\beta)\left\|e^{b.}y^{-\gamma}\right\|_{\alpha\mu\textrm{-H\"ol};T}
 \int_{t_{i}^{n}}^{t_{i + 1}^{n}}
 (t_{i + 1}^{n} - s)^{\alpha\mu}ds
 \nonumber\\
 & \leqslant &
 \label{EULmaj2}
 \frac{a(1 -\beta)}{\alpha\mu + 1}T^{\alpha\mu + 1}\left\|e^{b.}y^{-\gamma}\right\|_{\alpha\mu\textrm{-H\"ol};T}
 \frac{1}{n^{\alpha\mu + 1}}.
\end{eqnarray}
On the other hand, for each integer $i$ between $0$ and $k - 1$, we consider the two following cases (which are almost symmetric) :
\begin{enumerate}
 \item Suppose that $y_{i + 1}^{n}\geqslant z_{i + 1}^{n}$. Then,
 \begin{displaymath}
 \left(y_{i + 1}^{n}\right)^{-\gamma} -
 \left(z_{i + 1}^{n}\right)^{-\gamma}
 \leqslant 0.
 \end{displaymath}
 Therefore,
 \begin{eqnarray*}
  |y_{i + 1}^{n} - z_{i + 1}^{n}| & = &
  y_{i + 1}^{n} - z_{i + 1}^{n}\\
  & = &
  y_{i}^{n} - z_{i}^{n} +
  \frac{a(1-\beta)T}{n}e^{bt_{i + 1}^{n}}\left[
  (y_{i + 1}^{n})^{-\gamma} -
  (z_{i + 1}^{n})^{-\gamma}\right] -
  \varepsilon_{i}^{n}\\
  & \leqslant &
  |y_{i}^{n} - z_{i}^{n}| + 
  |\varepsilon_{i}^{n}|.
 \end{eqnarray*}
 \item Suppose that $z_{i + 1}^{n} > y_{i + 1}^{n}$. Then,
 \begin{displaymath}
 \left(z_{i + 1}^{n}\right)^{-\gamma} -
 \left(y_{i + 1}^{n}\right)^{-\gamma} < 0.
 \end{displaymath}
 Therefore,
 \begin{eqnarray*}
  |z_{i + 1}^{n} - y_{i + 1}^{n}| & = &
  z_{i + 1}^{n} - y_{i + 1}^{n}\\
  & = &
  z_{i}^{n} - y_{i}^{n} +
  \frac{a(1-\beta)T}{n}e^{bt_{i + 1}^{n}}\left[
  (z_{i + 1}^{n})^{-\gamma} -
  (y_{i + 1}^{n})^{-\gamma}\right] +
  \varepsilon_{i}^{n}\\
  & \leqslant &
  |y_{i}^{n} - z_{i}^{n}| + 
  |\varepsilon_{i}^{n}|.
 \end{eqnarray*}
\end{enumerate}
By putting these cases together :
\begin{equation}\label{EULDmaj3}
\forall i = 0,\dots, k - 1\textrm{, }
|z_{i + 1}^{n} - y_{i + 1}^{n}|
\leqslant
|z_{i}^{n} - y_{i}^{n}| + 
|\varepsilon_{i}^{n}|.
\end{equation}
By applying (\ref{EULDmaj3}) recursively from $k - 1$ down to $0$ :
\begin{eqnarray}
 |y_{k}^{n} - z_{k}^{n}| & \leqslant &
 |y_0 - z_0| +
 \sum_{i = 0}^{k - 1}
 |\varepsilon_{i}^{n}|
 \nonumber\\
 \label{EULmaj3}
 & \leqslant &
 \frac{a(1 -\beta)}{\alpha\mu + 1}T^{\alpha\mu + 1}\left\|e^{b.}y^{-\gamma}\right\|_{\alpha\mu\textrm{-H\"ol};T}
 \frac{1}{n^{\alpha\mu}}
 \xrightarrow[n\rightarrow\infty]{} 0
\end{eqnarray}
because $y_0 = z_0$ and by inequality (\ref{EULmaj2}).
\\
\\
Moreover, from inequality (\ref{EULmaj3}), there exists $N\in\mathbb N^*$ such that for every integer $n > N$,
\begin{displaymath}
|y_{k + 1}^{n} -
z_{k + 1}^{n}|\leqslant
\max_{i = 1,\dots,n}|y_{i}^{n} - z_{i}^{n}|\leqslant
m_y
\end{displaymath}
where,
\begin{displaymath}
m_y =
\frac{1}{2}\min_{s\in [0,T]} y_s.
\end{displaymath}
In particular,
\begin{displaymath}
y_{k + 1}^{n}\geqslant z_{k + 1}^{n} - m_y\geqslant m_y.
\end{displaymath}
Then $(y_{k + 1}^{n})^{-\gamma}\leqslant m_{y}^{-\gamma}$, and
\begin{eqnarray*}
 |y_{t}^{n} - y_{k}^{n}| & = &
 |y_{k + 1}^{n} - y_{k}^{n}|\frac{t - t_{k}^{n}}{t_{k + 1}^{n} - t_{k}^{n}}\\
 & \leqslant &
 \left[a(1-\beta)Te^{bT}
 m_{y}^{-\gamma} +
 T^{\alpha}\|\tilde w\|_{\alpha\textrm{-H\"ol};T}\right]\frac{1}{n^{\alpha}}
 \xrightarrow[n\rightarrow\infty]{} 0.
\end{eqnarray*}
In conclusion, from inequality (\ref{EULmaj1}) :
\begin{eqnarray}
 \label{EULmaj4}
 |y_{t}^{n} - y_t| & \leqslant &
 \left[a(1-\beta)Te^{bT}
 m_{y}^{-\gamma} +
 T^{\alpha}\|\tilde w\|_{\alpha\textrm{-H\"ol};T} +  \|y\|_{\alpha\textrm{-H\"ol};T}\right]\frac{1}{n^{\alpha}} +\\
 & &
 \frac{a(1 -\beta)}{\alpha\mu + 1}T^{\alpha\mu + 1}\left\|e^{b.}y^{-\gamma}\right\|_{\alpha\mu\textrm{-H\"ol};T}
 \frac{1}{n^{\alpha\mu}}
 \xrightarrow[n\rightarrow\infty]{} 0
  \nonumber.
\end{eqnarray}
That achieves the proof because the right hand side of inequality (\ref{EULmaj4}) is not depending on $k$ and $t$.
\end{proof}
\noindent
Finally, for every $n\in\mathbb N^*$ and $t\in [0,T]$, consider $x_{t}^{n} = e^{-bt}(y_{t}^{n})^{\gamma + 1}$.
\\
\\
The following corollary shows that $(x^n,n\in\mathbb N^*)$ is a converging approximation for $x = \tilde\pi(0,x_0;w)$ with $x_0 > 0$. Moreover, as the Euler approximation, it is just necessary to know $x_0$, $w$ and, parameters $a,b,\sigma$ and $\beta > 1 - \alpha$ to approximate the whole path $x$ by $x^n$ :
%

% Corollary : A converging approximation.

%
\begin{corollary}\label{CVapx}
Under assumptions \ref{Eass} and \ref{HYPw}, for $a > 0$ and $b\geqslant 0$, $(x^n,n\in\mathbb N^*)$ is uniformly converging on $[0,T]$ to $x$ with rate $n^{-\alpha\min(1,\gamma)}$.
\end{corollary}
%

% Proof.

%
\begin{proof}
For a given initial condition $x_0 > 0$, it has been shown that $x = e^{-b.}y^{\gamma + 1}$ is the solution of equation (\ref{DETmre}) by putting $y_0 = x_{0}^{1 - \beta}$, where $y$ is the solution of equation (\ref{EDS4}) with initial condition $y_0$.
\\
\\
From Theorem \ref{EULlim} :
\begin{eqnarray*}
 \|x - x^n\|_{\infty;T} & \leqslant &
 C\|y - y^n\|_{\infty;T}\\
 & \leqslant &
 C\left[a(1-\beta)Te^{bT}
 m_{y}^{-\gamma} +
 T^{\alpha}\|\tilde w\|_{\alpha\textrm{-H\"ol};T} +  \|y\|_{\alpha\textrm{-H\"ol};T}\right]\frac{1}{n^{\alpha}} +\\
 & &
 C\frac{a(1 -\beta)}{\alpha\mu + 1}T^{\alpha\mu + 1}\left\|e^{b.}y^{-\gamma}\right\|_{\alpha\mu\textrm{-H\"ol};T}
 \frac{1}{n^{\alpha\mu}}
 \xrightarrow[n\rightarrow\infty]{} 0
\end{eqnarray*}
where, $C$ is the Lipschitz constant of $s\mapsto s^{\gamma + 1}$ on
\begin{displaymath}
\left[0,\|y\|_{\infty;T} + \sup_{n\in\mathbb N^*}\|y^n\|_{\infty;T}\right].
\end{displaymath}
Then, $(x^n,n\in\mathbb N^*)$ is uniformly converging to $x$ with rate $n^{-\alpha\min(1,\gamma)}$.
\end{proof}
\noindent
\textbf{Remark.} When $\alpha > 1/2$ ; $\beta > 1-\alpha > 1/2$ and then $\gamma > 1$. Therefore, $(x^n,n\in\mathbb N^*)$ is uniformly converging with rate $n^{-\alpha} < n^{1 - 2\alpha}$. In other words, the approximation of Corollary \ref{CVapx} converges faster than the classic Euler approximation for equations satisfying assumptions of \cite{LEJ10}, Propositions 5. It is related to the specific form of the vector field $V$.
%

% Section : Probabilistic properties of the generalized mean-reverting equation.

%
\section{Probabilistic properties of the generalized mean-reverting equation}
\noindent
Consider the Gaussian process $W$ and the probability space $(\Omega,\mathcal A,\mathbb P)$ introduced at Section 2. Under Assumption \ref{HYPW}, almost every paths of $W$ are satisfying Assumption \ref{HYPw}. Then, under assumptions \ref{Eass} and \ref{HYPW}, results of Section 3 hold true for $\tilde\pi_V(0,x_0;W)$, with deterministic initial condition $x_0 > 0$.
\\
\\
This section is essentially devoted to complete them on probabilistic side. In particular, we prove that $\tilde\pi_V(0,x_0;W)$ belongs to $L^p(\Omega)$ for every $p\geqslant 1$. We also show that the approximation introduced at Section 3 for $\tilde\pi_V(0,x_0;W)$ is converging in $L^p(\Omega)$ for every $p\geqslant 1$.
\\
\\
\textbf{Remark.} Since $W$ is a $1$-dimensional process, as mentioned at Section 2, there exists an explicit geometric $1/\alpha$-rough path $\mathbb W$ over it, matching with the enhanced Gaussian process provided by P. Friz and N. Victoir at \cite{FV08}, Theorem 15.33. That explains why Assumption \ref{HYPW} is sufficient to extend deterministic results of Section 3 to $\tilde\pi_V(0,x_0;W)$.
%

% Subsection : Extension of existence results and properties of the solution's distribution.

%
\subsection{Extension of existence results and properties of the solution's distribution}
On one hand, when $\beta\not\in ]1-\alpha,1]$, Proposition \ref{MRprob} extend remark 2 of Proposition \ref{MRres} on probabilistic side. On the other hand, we study properties of the distribution of $X =\tilde\pi_V(0,x_0;W)$ defined on $\mathbb R_+$, when $W = (W_t,t\in\mathbb R_+)$ is a $1$-dimensional Gaussian process with locally $\alpha$-H\"older continuous paths, stationary increments and satisfies a self-similar property.
%

% Proposition : Existence of solutions when $\beta < 1-\alpha$.

%
\begin{proposition}\label{MRprob}
Consider $a > 0$, $b\geqslant 0$, $\alpha\in ]0,1[$, a process $W$ satisfying Assumption \ref{HYPW}, $x_0 > 0$, $y_0 = x_{0}^{1-\beta}$,
\begin{displaymath}
\sigma^2 =
\sup_{t\in [0,T]}
\mathbb E\left(\tilde W_{t}^{2}\right)
\end{displaymath}
and
\begin{displaymath}
A = \left\{\textrm{$\tilde\pi_V(0,x_0;W)$ is defined on $[0,T]$}\right\}.
\end{displaymath}
If $2\sigma^2\ln(2) < y_{0}^{2}$, then $\mathbb P(A) > 0$.
\end{proposition}
%

% Proof.

%
\begin{proof}
On one hand, by Remark 2 of Proposition \ref{MRres} :
\begin{eqnarray*}
 A & \supset & \{\inf_{t\in [0,T]}\tilde W_t > -y_0\}\\
 & = &
 \{\sup_{t\in [0,T]} -\tilde W_t < y_0\}.
\end{eqnarray*}
On the other hand, since $-\tilde W$ is a $1$-dimensional centered Gaussian process with continuous paths by construction, by Borell's inequality (cf. \cite{ADLER90}, Theorem 2.1) :
\begin{displaymath}
\mathbb P\left(
\sup_{t\in [0,T]}
-\tilde W_t > y_0\right)
\leqslant
2\exp\left(-\frac{y_{0}^{2}}{2\sigma^2}\right)
\end{displaymath}
with $\sigma^2 < \infty$. Therefore,
\begin{eqnarray*}
 \mathbb P(A) & \geqslant &
 1 - \mathbb P\left(
 \sup_{t\in [0,T]}
 -\tilde W_t > y_0\right)\\
 & \geqslant &
 1 - 2\exp\left(-\frac{y_{0}^{2}}{2\sigma^2}\right) > 0.
\end{eqnarray*}
\end{proof}
%

% Proposition : Consequence of increments stationarity of $W$.

%
\begin{proposition}\label{EDS3si}
Assume that $W = (W_t,t\in\mathbb R_+)$ is a $1$-dimensional centered Gaussian process with locally $\alpha$-H\"older continuous paths, and there exists $h > 0$ such that :
\begin{displaymath}
W_{. + h} - W_h
\stackrel{\mathcal D}{=}
W.
\end{displaymath}
Under Assumption \ref{Eass}, for $a > 0$ and $b\geqslant 0$, with any deterministic initial condition $x_0 > 0$ :
\begin{displaymath}
\tilde\pi_{V;0,t + h}(0,x_0;W)
\stackrel{\mathcal D}{=}
\tilde\pi_{V;0,t}(0,X_h;W)
\end{displaymath}
for every $t\in\mathbb R_+$.
\end{proposition}
%

% Proof.

%
\begin{proof}
By Proposition \ref{MRres}, $X$ has almost surely continuous and strictly positive paths on $\mathbb R_+$. Then, by Theorem \ref{CVformula} applied to almost every paths of $X$ and to the map $u\mapsto u^{1-\beta}$ between $0$ and $t\in\mathbb R_+$ :
\begin{displaymath}
X_{t}^{1-\beta} =
x_{0}^{1-\beta} +
(1-\beta)\int_{0}^{t}
X_{u}^{-\beta}(a - bX_u)du +
\sigma(1-\beta)W_t.
\end{displaymath}
Therefore, $X_{. + h}^{1-\beta}\stackrel{\mathcal D}{=} Z(h)$ where,
\begin{displaymath}
Z_t(h) =
X_{h}^{1-\beta} +
(1-\beta)\int_{0}^{t}Z_{u}^{-\gamma}(h)
\left[a -bZ_{u}^{\gamma + 1}(h)\right]du +
\sigma(1-\beta)W_t
\textrm{ ; }
t\in\mathbb R_+
\end{displaymath}
because $W_{. + h} - W_h\stackrel{\mathcal D}{=}W$.
\\
\\
In conclusion, by applying Theorem \ref{CVformula} to almost every paths of $Z(h)$ and to the map $u\mapsto u^{\gamma + 1}$ :
\begin{displaymath}
X_{t + h} - X_h
\stackrel{\mathcal D}{=}
\int_{0}^{t}
\left(a - bX_{u + h}\right)du +
\sigma\int_{0}^{t}X_{u + h}^{\beta}dW_u
\end{displaymath}
for every $t\in\mathbb R_+$.
\end{proof}
%

% Proposition : Consequence of self-similarity of $W$.

%
\begin{proposition}\label{EDS3ss}
Assume that $W = (W_t,t\in\mathbb R_+)$ is a $1$-dimensional centered Gaussian process with locally $\alpha$-H\"older continuous paths, and there exists $h > 0$ such that :
\begin{displaymath}
\forall\varepsilon > 0\textrm{, }
W_{\varepsilon .}
\stackrel{\mathcal D}{=}
\varepsilon^hW.
\end{displaymath}
Under Assumption \ref{Eass}, for $a > 0$ and $b\geqslant 0$, with any deterministic initial condition $x_0 > 0$ :
\begin{displaymath}
\tilde\pi_{V;0,\varepsilon t}
\left(0,x_0;W\right)
\stackrel{\mathcal D}{=}
\tilde\pi_{V_{\varepsilon,h};0,t}\left(0,x_0;W\right)
\end{displaymath}
for every $t\in\mathbb R_+$ and $\varepsilon > 0$, with :
\begin{displaymath}
\forall x\in\mathbb R_+\textrm{, }
\forall t,w\in\mathbb R\textrm{, }
V_{\varepsilon,h}(x).(t,w) =
\varepsilon (a - bx)t +\sigma\varepsilon^hx^{\beta}w.
\end{displaymath}
\end{proposition}
%

% Proof.

%
\begin{proof}
By Proposition \ref{MRres}, $X$ has almost surely continuous and strictly positive paths on $\mathbb R_+$. Then, by Theorem \ref{CVformula} applied to almost every paths of $X$ and to the map $u\mapsto u^{1-\beta}$ between $0$ and $t\in\mathbb R_+$ :
\begin{displaymath}
X_{t}^{1-\beta} =
x_{0}^{1-\beta} +
(1-\beta)\int_{0}^{t}
X_{u}^{-\beta}(a - bX_u)du +
\sigma(1-\beta)W_t.
\end{displaymath}
Therefore, for every $\varepsilon > 0$, $X_{\varepsilon .}^{1-\beta}\stackrel{\mathcal D}{=} Z(\varepsilon)$ where,
\begin{displaymath}
Z_t(\varepsilon) =
x_{0}^{1 -\beta} +
\varepsilon(1-\beta)\int_{0}^{t}Z_{u}^{-\gamma}(\varepsilon)
\left[a -bZ_{u}^{\gamma + 1}(\varepsilon)\right]du +
\varepsilon^h\sigma(1-\beta)W_t
\textrm{ ; }
t\in\mathbb R_+
\end{displaymath}
because $W_{\varepsilon .}\stackrel{\mathcal D}{=}\varepsilon^hW$.
\\
\\
In conclusion, by applying Theorem \ref{CVformula} to almost every paths of $Z(\varepsilon)$ and to the map $u\mapsto u^{\gamma + 1}$ :
\begin{displaymath}
X_{\varepsilon t}
\stackrel{\mathcal D}{=}
x_0 +
\varepsilon\int_{0}^{t}
\left(a - bX_{\varepsilon u}\right)du +
\sigma\varepsilon^h\int_{0}^{t}X_{\varepsilon u}^{\beta}dW_u
\end{displaymath}
for every $t\in\mathbb R_+$ and $\varepsilon > 0$.
\end{proof}
\noindent
\textbf{Remark.} Typically, mean-reverting equations driven by a fractional Brownian motion are concerned by propositions \ref{EDS3si} and \ref{EDS3ss}.
%

% Proposition : Hitting times for fractional Brownian signal.

%
\begin{proposition}\label{HTfbm}
Consider $a > 0$, $b\geqslant 0$ and a $1$-dimensional fractional Brownian motion $(B_{t}^{H},t\in\mathbb R_+)$ with Hurst parameter $H\in ]0,1[$. Under Assumption \ref{Eass}, for every $\varepsilon > 0$ ($x_0 > 0$) :
\begin{displaymath}
\tau_{\varepsilon}^{4} =
\inf\left\{
t\geqslant 0 : \tilde\pi_V(0,x_0;B^H)_t = \varepsilon
\right\} < \infty
\textrm{ $\mathbb P$-p.s.}
\end{displaymath}
\end{proposition}
%

% Proof.

%
\begin{proof}
Consider $\varepsilon > 0$ and
\begin{displaymath}
\tau_{\varepsilon}^{5} =
\inf\left\{t\geqslant 0 : Z_t =\varepsilon\right\}
\end{displaymath}
where, $Z = \tilde\pi_{V}^{1-\beta}(0,x_0;B^H)$.
\\
\\
\textbf{Case 1} ($\varepsilon\leqslant x_{0}^{1-\beta}$). On one hand, since $\tau_{\varepsilon}^{5} = \infty$ if and only if $Z_t > \varepsilon$ for every $t\in\mathbb R_+$, and
\begin{displaymath}
Z_t =
Z_0 +
(1-\beta)
\int_{0}^{t}
(aZ_{s}^{-\gamma} - bZ_s)ds +
\sigma(1 - \beta)B_{t}^{H},
\end{displaymath}
then $\tau_{\varepsilon}^{5} = \infty$ implies that :
\begin{displaymath}
\forall t\in\mathbb R_+
\textrm{, }
B_{t}^{H}\geqslant
\frac{(1-\beta)(b\varepsilon - \varepsilon^{-\gamma}a)t +\varepsilon - Z_0}{\sigma(1-\beta)}.
\end{displaymath}
Therefore,
\begin{eqnarray*}
 \mathbb P(\tau_{\varepsilon}^{5} = \infty)
 & \leqslant &
 \mathbb P\left[
 \forall t\in\mathbb R_+
 \textrm{, }
 B_{t}^{H}\geqslant
\frac{(1-\beta)(b\varepsilon - \varepsilon^{-\gamma}a)t +\varepsilon - Z_0}{\sigma(1-\beta)}\right]\\
 & \leqslant &
 \mathbb P\left[
 B_{t}^{H}\geqslant
\frac{(1-\beta)(b\varepsilon - \varepsilon^{-\gamma}a)t +\varepsilon - Z_0}{\sigma(1-\beta)}\right]
\end{eqnarray*}
for every $t\in\mathbb R_+$.
\\
\\
On the other hand, since $B_{t}^{H}\rightsquigarrow\mathcal N(0,t^{2H})$ :
\begin{displaymath}
\mathbb P\left[
B_{t}^{H}\geqslant
\frac{(1-\beta)(b\varepsilon - \varepsilon^{-\gamma}a)t +\varepsilon - Z_0}{\sigma^2(1-\beta)^2}\right] =
\frac{1}{t^H\sqrt{2\pi}}
\int_{0}^{\infty}
\varphi(\xi,t)d\xi
\end{displaymath}
with
\begin{displaymath}
\varphi(\xi,t) =
\exp\left[-
\frac{\left[\xi + (1-\beta)(b\varepsilon - \varepsilon^{-\gamma}a)t +\varepsilon - Z_0\right]^2}{\sigma^2(1-\beta)^2t^{2H}}
\right].
\end{displaymath}
For every $\xi\in\mathbb R$ and every $\varepsilon > 0$,
\begin{displaymath}
\lim_{t\rightarrow\infty}
\varphi(\xi,t) =
\lim_{t\rightarrow\infty}
\exp\left[-
\frac{\left[\xi + (1-\beta)(b\varepsilon - \varepsilon^{-\gamma}a)\right]^2}{\sigma^2(1-\beta)^2}t^{2(1 - H)}
\right] =
0,
\end{displaymath}
and $t\in\mathbb R_{+}^{*}\mapsto\varphi(\xi,t)$ is a continuous, decreasing map. Then, for every $t\geqslant 1$,
\begin{displaymath}
|\varphi(\xi,t)|
\leqslant
|\varphi(\xi,1)|
\sim_{\xi\rightarrow\infty}
\exp\left[
-\frac{\xi^2}{\sigma^2(1-\beta)^2}
\right]
\in L^1(\mathbb R ; d\xi).
\end{displaymath}
Therefore, by Lebesgue's theorem :
\begin{displaymath}
\lim_{t\rightarrow\infty}
\mathbb P\left[
B_{t}^{H}\geqslant
\frac{(1-\beta)(b\varepsilon - \varepsilon^{-\gamma}a)t +\varepsilon - Z_0}{\sigma(1-\beta)}\right] = 0,
\end{displaymath}
and for every $\varepsilon\in ]0,x_{0}^{1-\beta}]$, $\tau_{\varepsilon}^{5} < \infty$ almost surely.
\\
\\
\textbf{Case 2} ($\varepsilon > x_{0}^{1-\beta}$). In that case, $\tau_{\varepsilon}^{5} = \infty$ if and only if, $0 < Z_t < \varepsilon$ for every $t\in\mathbb R_+$. Then, with ideas of the first case :
\begin{eqnarray*}
 \mathbb P(\tau_{\varepsilon}^{5} = \infty)
 & \leqslant &
 \mathbb P\left[
 B_{t}^{H}\leqslant
\frac{(1-\beta)(b\varepsilon - \varepsilon^{-\gamma}a)t +\varepsilon - Z_0}{\sigma(1-\beta)}\right]\\
 & \leqslant &
 \frac{1}{t^H\sqrt{2\pi}}
 \int_{-\infty}^{0}\varphi(\xi,t)d\xi
\end{eqnarray*}
for every $t\in\mathbb R_+$.
\\
\\
Moreover, results on $\varphi$ have been established for every $\xi\in\mathbb R$ and every $\varepsilon > 0$ at case 1 then, by Lebesgue's theorem :
\begin{displaymath}
\lim_{t\rightarrow\infty}
\mathbb P\left[
B_{t}^{H}\leqslant
\frac{(1-\beta)(b\varepsilon - \varepsilon^{-\gamma}a)t +\varepsilon - Z_0}{\sigma(1-\beta)}\right] = 0,
\end{displaymath}
and for every $\varepsilon > x_{0}^{1-\beta}$, $\tau_{\varepsilon}^{5} < \infty$ almost surely.
\\
\\
In conclusion, since $\tau_{\varepsilon}^{4} = \tau_{\varepsilon^{1-\beta}}^{5}$ by Lemma \ref{EDS3change}, for every $\varepsilon > 0$, $\tau_{\varepsilon}^{4} < \infty$ almost surely.
\end{proof}
%

% Subsection : Integrability and convergence results.

%
\subsection{Integrability and convergence results}
Consider the implicit Euler approximation $(Y^n,n\in\mathbb N^*)$ for the following SDE :
\begin{displaymath}
Y_t =
y_0 +
a(1-\beta)
\int_{0}^{t}
Y_{s}^{-\gamma}e^{bs}ds +
\tilde W_t
\textrm{ ; }
t\in [0,T]
\textrm{, }
y_0 > 0
\end{displaymath}
where,
\begin{displaymath}
\tilde W_t =
\int_{0}^{t}\vartheta_sdW_s
\textrm{ and }
\vartheta_t =
\sigma(1-\beta)e^{b(1-\beta)t}
\end{displaymath}
for every $t\in [0,T]$.
%

% Proposition : Integrability results.

%
\begin{proposition}\label{INTsol_eul}
Under assumptions \ref{Eass} and \ref{HYPW}, for $a > 0$ and $b\geqslant 0$, with any deterministic initial condition $x_0 > 0$ :
\begin{enumerate}
 \item $\|\tilde\pi_V(0,x_0;W)\|_{\infty;T}$ belongs to $L^p(\Omega)$ for every $p\geqslant 1$.
 \item For every $p\geqslant 1$,
 \begin{displaymath}
 \sup_{n\in\mathbb N^*}
 \left\|X^n\right\|_{\infty;T}
 \in L^p(\Omega)
 \end{displaymath}
 where, for every $n\in\mathbb N^*$, $X^n = e^{-b.}(Y^n)^{\gamma + 1}$ with $y_0 = x_{0}^{1 -\beta}$.
\end{enumerate}
\end{proposition}
%

% Proof.

%
\begin{proof}
On one hand, by Proposition \ref{MRint} and Fernique's theorem :
\begin{eqnarray*}
 \left\|\tilde\pi_V(0,x_0;W)\right\|_{\infty;T}
 & \leqslant &
 \left[
 x_{0}^{1-\beta} + a(1-\beta)e^{bT}x_{0}^{-\beta}T +\right.\\
 & &
 \left.
 \sigma(b\vee 2)(1-\beta)(1 + T)e^{b(1-\beta)T}\|W\|_{\infty;T}
 \right]^{\gamma + 1}\in L^p(\Omega)
\end{eqnarray*}
for every $p\geqslant 1$.
\\
\\
On the other hand, by Lemma \ref{EULcv} and Fernique's theorem :
\begin{eqnarray*}
 \sup_{n\in\mathbb N^*}
 \left\|Y^n\right\|_{\infty;T}
 & \leqslant &
 y_0 + a(1-\beta)e^{bT}y_{0}^{-\gamma}T +\\
 & &
 \sigma(b\vee 2)(1-\beta)(1 + T)e^{b(1-\beta)T}\|W\|_{\infty;T}\in
 L^q(\Omega)
\end{eqnarray*}
for every $q\geqslant 1$. Then, by putting $q = (\gamma + 1)p$ for every $p\geqslant 1$,
\begin{displaymath}
\sup_{n\in\mathbb N^*}
\left\|X^n\right\|_{\infty;T}
\in L^p(\Omega).
\end{displaymath}
\end{proof}
%

% Corollary : Probabilistic convergence of $X^n$.

%
\begin{corollary}\label{CVXprob}
Under assumptions \ref{Eass} and \ref{HYPW}, for $a > 0$ and $b\geqslant 0$, with any deterministic initial condition $x_0 > 0$, $(X^n,n\in\mathbb N^*)$ is uniformly converging on $[0,T]$ to $\tilde\pi_V(0,x_0;w)$ in $L^p(\Omega)$ for every $p\geqslant 1$.
\end{corollary}
%

% Proof.

%
\begin{proof}
By Corollary \ref{CVapx} :
\begin{displaymath}
\left\|X^n - \tilde\pi_V(0,x_0;W)\right\|_{\infty;T}
\xrightarrow[n\rightarrow\infty]{\mathbb P\textrm{-a.s.}}
0.
\end{displaymath}
Then, by Proposition \ref{INTsol_eul} and Vitali's convergence theorem, $(X^n,n\in\mathbb N^*)$ is uniformly converging to $ \tilde\pi_V(0,x_0;W)$ in $L^p(\Omega)$ for every $p\geqslant 1$.
\end{proof}
\noindent
\textbf{Remark.} Note that Proposition \ref{INTsol_eul} is crucial to ensure this convergence in $L^p(\Omega)$ for every $p\geqslant 1$. Indeed, inequality (\ref{EULmaj4}) doesn't allow to conclude because it is not sure that $\|e^{b.}Y^{-\gamma}\|_{\alpha\mu\textrm{-H\"ol};T}\in L^1(\Omega)$.
%

% Subsection : A large deviation principle for the generalized mean-reverting equation.

%
\subsection{A large deviation principle for the generalized M-R equation}
We establish a large deviation principle for the generalized mean-reverting equation (as P. Friz and N. Victoir at \cite{FV08}, Section 19.4).
\\
\\
First of all, let's remind basics on large deviations (for details, the reader can refer to \cite{DZ98}).
\\
\\
Throughout this subsection, assume that $\inf(\emptyset) = \infty$.
%

% Definition : LDP.

%
\begin{definition}\label{LDP}
Let $E$ be a topological space and let $I : E\rightarrow [0,\infty]$ be a good rate function (i.e. a lower semicontinuous map such that $\{x\in E : I(x)\leqslant\lambda\}$ is a compact subset of $E$ for every $\lambda\geqslant 0$).
\\
\\
A family $(\mu_{\varepsilon},\varepsilon > 0)$ of probability measures on $(E,\mathcal B(E))$ satisfies a large deviation principle with good rate function $I$ if and only if, for every $A\in\mathcal B(E)$,
\begin{displaymath}
-I(A^{\circ})\leqslant
\underline{\lim}_{\varepsilon\rightarrow 0}
\varepsilon\log\left[\mu_{\varepsilon}(A)\right]\leqslant
\overline{\lim}_{\varepsilon\rightarrow 0}
\varepsilon\log\left[\mu_{\varepsilon}(A)\right]\leqslant
-I(\bar A)
\end{displaymath}
where,
\begin{displaymath}
\forall A\in\mathcal B(E)\textrm{, }
I(A) =
\inf_{x\in A}
I(x).
\end{displaymath}
\end{definition}
%

% Proposition : Contraction principle.

%
\begin{proposition}\label{CPrinciple}
Consider $E$ and $F$ two Hausdorff topological spaces, a continuous map $f : E\rightarrow F$ and a family $(\mu_{\varepsilon},\varepsilon > 0)$ of probability measures on $(E,\mathcal B(E))$.
\\
\\
If $(\mu_{\varepsilon},\varepsilon > 0)$ satisfies a large deviation principle with good rate function $I : E\rightarrow [0,\infty]$, then $(\mu_{\varepsilon}\circ f^{-1},\varepsilon > 0)$ satisfies a large deviation principle on $(F,\mathcal B(F))$ with good rate function $J : F\rightarrow [0,\infty]$ such that :
\begin{displaymath}
J(y) =
\inf\left\{I(x) ; x\in E\textrm{ and }f(x) = y\right\}
\end{displaymath}
for every $y\in F$.
\end{proposition}
\noindent
That result is called contraction principle. The reader can refer to \cite{DZ98}, Lemma 4.1.6 for a proof.
\\
\\
Consider the space $C^{0,\alpha}([0,T];\mathbb R)$ of functions $\varphi\in C^{\alpha\textrm{-H\"ol}}([0,T];\mathbb R)$ such \mbox{that :}
\begin{displaymath}
\lim_{\delta\rightarrow 0^+}
\omega_{\varphi}(\delta) = 0
\textrm{ with }
\omega_{\varphi}(\delta) =
\sup_{
\begin{tiny}
\begin{array}{rcl}
(s,t) & \in & \Delta_T\\
|t - s| & \leqslant & \delta
\end{array}
\end{tiny}}
\frac{|\varphi(t) -\varphi(s)|}{|t - s|^{\alpha}}
\end{displaymath}
for every $\delta > 0$.
\\
\\
In the sequel, $C^{0,\alpha}([0,T];\mathbb R)$ is equipped with $\|.\|_{\alpha\textrm{-H\"ol};T}$ and the Borel $\sigma$-field generated by open sets of the $\alpha$-H\"older topology. The same way, $C^0([0,T];\mathbb R)$ is equipped with $\|.\|_{\infty;T}$ and the Borel $\sigma$-field generated by open sets of the uniform topology.
\\
\\
Now, suppose that $W$ satisfies :
%

% Assumption on $W$ for LDP.

%
\begin{assumption}\label{WSass}
There exists $h > 0$ such that :
\begin{displaymath}
\forall\varepsilon > 0\textrm{, }
W_{\varepsilon .}
\stackrel{\mathcal D}{=}
\varepsilon^hW.
\end{displaymath}
Moreover, $\mathcal H_{W}^{1}\subset C^{0,\alpha}([0,T];\mathbb R)$ and $(C^{0,\alpha}([0,T];\mathbb R),\mathcal H_{W}^{1},\mathbb P)$ is an abstract Wiener space.
\end{assumption}
\noindent
\textbf{Remarks :}
\begin{enumerate}
 \item The notion of abstract Wiener space is defined and detailed in \mbox{M. Ledoux \cite{LED94}.}
 \item Typically, the fractional Brownian motion with Hurst parameter $H > 1/4$ satisfies Assumption \ref{WSass} (cf. \cite{SSTT07}, Proposition 4.1).
\end{enumerate}
Consider the stochastic differential equation :
\begin{equation}\label{EDSldp}
X_t =
x_0 +
\frac{1}{\delta}
\int_{0}^{t}
\left(a - bX_s\right)ds +
\frac{\sigma}{\delta^{h - 1}}
\int_{0}^{t}
X_{s}^{\beta}dW_s
\textrm{ ; }
t\in [0,T]
\end{equation}
\noindent
where, $x_0 > 0$ is a deterministic initial condition, $a,b,\sigma,\delta > 0$ and $\beta\in ]0,1]$ satisfies Assumption \ref{Eass}.
\\
\\
Under assumptions \ref{Eass} and $\ref{HYPW}$, by propositions \ref{MRres} and \ref{INTsol_eul}, equation (\ref{EDSldp}) admits a unique solution belonging to $L^p(\Omega)$ for every $p\geqslant 1$.
\\
\\
Moreover, under Assumption \ref{WSass}, by Proposition \ref{EDS3ss} :
\begin{equation}\label{EDSTRldp}
X_{\varepsilon t} =
x_0 +
\frac{\varepsilon}{\delta}
\int_{0}^{t}
\left(a - bX_{\varepsilon s}\right)ds +
\frac{\sigma\varepsilon^h}{\delta^{h - 1}}
\int_{0}^{t}
X_{\varepsilon s}^{\beta}dW_s
\end{equation}
for every $t\in [0,T]$ and $\varepsilon > 0$.
\\
\\
In the sequel, assume that $\delta =\varepsilon$. Then, $X_{\varepsilon .}$ satisfies :
\begin{displaymath}
X_{\varepsilon .} =
\tilde\pi_V\left(0,x_0;\varepsilon W\right)
\end{displaymath}
where, $V$ is the map defined on $\mathbb R_+$ by :
\begin{displaymath}
\forall x\in\mathbb R_+\textrm{, }
\forall t,w\in\mathbb R\textrm{, }
V(x).(t,w) =
(a - bx)t +
\sigma x^{\beta}w.
\end{displaymath}
Let show that $(X_{\varepsilon .},\varepsilon > 0)$ satisfies a large deviation principle :
%

% Proposition : LDP for $(X_{\varepsilon .},\varepsilon > 0)$.

%
\begin{proposition}\label{LDPx}
Consider $x_0 > 0$. Under assumptions \ref{Eass}, \ref{HYPW} and \ref{WSass}, for $a > 0$ and $b\geqslant 0$, $(X_{\varepsilon .},\varepsilon > 0)$ satisfies a large deviation principle on $C^0([0,T];\mathbb R)$ with good rate function $J : C^0([0,T];\mathbb R)\rightarrow [0,\infty]$ defined by :
\begin{displaymath}
\forall y\in C^0([0,T];\mathbb R)
\textrm{, }
J(y) =
\inf\left\{
I(w) ;
w\in C^{0,\alpha}([0,T];\mathbb R)
\textrm{ and }
y = \tilde\pi_V(0,x_0;w)
\right\}
\end{displaymath}
where,
\begin{displaymath}
I(w) =
\left\{
\begin{array}{rcl}
\frac{1}{2}\|w\|_{\mathcal H_{W}^{1}} & \textrm{if} & w\in\mathcal H_{W}^{1}\\
 \infty & \textrm{if} & w\not\in\mathcal H_{W}^{1}
\end{array}
\right.
\end{displaymath}
for every $w\in C^{0,\alpha}([0,T];\mathbb R)$.
\end{proposition}
%

% Proof.

%
\begin{proof}
Since $C^{0,\alpha}([0,T];\mathbb R)\subset C^{\alpha\textrm{-H\"ol}}([0,T];\mathbb R)$ by construction, Proposition \ref{MRcont} implies that $\tilde\pi_V(0,x_0;.)$ is continuous from
\begin{displaymath}
C^{0,\alpha}([0,T];\mathbb R)
\textrm{ into }
C^0([0,T];\mathbb R).
\end{displaymath}
On the other hand, under Assumption \ref{WSass}, by M. Ledoux \cite{LED94}, Theorem 4.5 ; $(\varepsilon W,\varepsilon > 0)$ satisfies a large deviation principle on $C^{0,\alpha}([0,T];\mathbb R)$ with good rate function $I$.
\\
\\
Therefore, since $X_{\varepsilon .} =\tilde\pi_V(0,x_0;\varepsilon W)$ for every $\varepsilon > 0$, by the contraction principle (Proposition \ref{CPrinciple}), $(X_{\varepsilon .},\varepsilon > 0)$ satisfies a large deviation principle on $C^0([0,T];\mathbb R)$ with good rate function $J$.
\end{proof}
%

% Subsection : Density with respect to the Lebesgue's measure for the solution. 

%
\subsection{Density with respect to Lebesgue's measure for the solution}
Via Bouleau-Hirsch's method, this subsection is devoted to show that $\tilde\pi_V(0,x_0;W)_t$ admits a density with respect to Lebesgue's measure on $(\mathbb R,\mathcal B(\mathbb R))$ for every $t\in ]0,T]$ and every $x_0 > 0$.
\\
\\
\textbf{Notation.} For two normed vector spaces $E$ and $F$, the embedment of $E$ in $F$ is denoted by $E\hookrightarrow F$.
\\
\\
Throughout this subsection, assume that $W$ satisfies :
%

% Assumption : Embedding of Carmeron-Martin space.

%
\begin{assumption}\label{DSNass}
Cameron-Martin's space of $W$ satisfies :
\begin{displaymath}
C_{0}^{\infty}([0,T];\mathbb R)\subset
\mathcal H_{W}^{1}\hookrightarrow
C^{\alpha\textrm{-H\"ol}}([0,T];\mathbb R).
\end{displaymath}
\end{assumption}
\noindent
\textbf{Example.} A fractional Brownian motion with Hurst parameter $H > 1/4$ satisfies Assumption \ref{DSNass}.
%

% Proposition : Bouleau-Hirsch's method for the generalized M-R equation.

%
\begin{proposition}\label{BHmre}
Under assumptions \ref{Eass}, \ref{HYPW} and \ref{DSNass}, for $a > 0$, $b\geqslant 0$ and any $t\in ]0,T]$, $\tilde\pi_V(0,x_0;W)_t$ admits a density with respect to Lebesgue's measure on $(\mathbb R,\mathcal B(\mathbb R))$.
\end{proposition}
%

% Proof.

%
\begin{proof}
With notations of Proposition \ref{ITOdiff}, by Proposition \ref{BHcondition} and the transfer theorem, it is sufficient to show that $\omega\in\Omega\mapsto z_t[z_0,W(\omega)]$ satisfies Bouleau-Hirsch's condition for any $t\in ]0,T]$.
\\
\\
On one hand, by Proposition \ref{ITOdiff} (cf. Proof), $z(z_0,.)$ is continuously differentiable from $C^{\alpha\textrm{-H\"ol}}([0,T];\mathbb R)$ into $C^0([0,T];\mathbb R)$. Then, $z(z_0,.)$ is continuously differentiable on
\begin{displaymath}
\mathcal H_{W}^{1}
\hookrightarrow
C^{\alpha\textrm{-H\"ol}}([0,T];\mathbb R)
\subset
C^0([0,T];\mathbb R).
\end{displaymath}
By P. Friz and N. Victoir \cite{FV08}, Lemma 15.58, for almost every $\omega\in\Omega$,
\begin{displaymath}
\forall h\in\mathcal H_{W}^{1}
\textrm{, }
W(\omega + h) =
W(\omega) + h.
\end{displaymath}
Therefore, almost surely :
\begin{displaymath}
z\left[x_0;W(. + h)\right] =
z\left[x_0;W(.) + h\right],
\end{displaymath}
and $z(x_0,W)$ is continuously $\mathcal H_{W}^{1}$-differentiable.
\\
\\
On the other hand, by Proposition \ref{ITOdiff}, for every $h\in\mathcal H_{W}^{1}$,
\begin{eqnarray*}
 D_hz_t(z_0,W) & = &
 \sigma(1-\beta)h_t +
 \int_{0}^{t}
 \dot F\left[z_s(z_0,W)\right]D_hz_s(z_0,W)ds\\
 & = &
 \sigma(1-\beta)\int_{0}^{t}h_s\exp\left[
 \int_{s}^{t}\dot F\left[z_u(z_0,W)\right]du
 \right]ds.
\end{eqnarray*}
In particular, $D_hz_t(z_0,W) > 0$ for $h := \textrm{Id}_{[0,T]}\in\mathcal H_{W}^{1}$.
\\
\\
In conclusion, by Proposition \ref{BHcondition}, for every $t\in ]0,T]$, $z_t(z_0,W)$ and then $\tilde\pi_V(0,x_0;W)_t$, admit a density with respect to Lebesgue's measure on $(\mathbb R,\mathcal B(\mathbb R))$ respectively.
\end{proof}
%

% Section : A pharmacokinetic model.

%
\section{A generalized mean-reverting pharmacokinetic model}
\noindent
We study a pharmacokinetic model based on a particular generalized mean-reverting equation (inspired by K. Kalogeropoulos et al. \cite{KDP08}).
\\
\\
In order to study the absorption/elimination processes of a given drug, the following deterministic mono-compartment model is classically used :
\begin{equation}\label{monocp}
C_t =
\int_{0}^{t}
\left(\frac{A_0K_a}{v}e^{-K_as} - K_eC_s\right)ds
\textrm{ ; }
t\in [0,T]
\end{equation}
where :
\begin{itemize}
 \item $A_0 > 0$ is the dose administered to the patient at initial time.
 \item $v > 0$ is the volume of the elimination compartment $E$ (extra-vascular tissues).
 \item $K_a\geqslant 0$ is the rate of absorption in compartment $A$. If the drug is administered by rapid injection, an IV bolus injection, it is natural to take $K_a = 0$.
 \item $K_e > 0$ is the rate of elimination in compartment $E$, describing removal of the drug by all elimination processes including excretion and metabolism.
 \item $C_t$ is the concentration of the drug in compartment $E$ at time $t\in [0,T]$.
\end{itemize}
\textbf{Remark.} About deterministic pharmacokinetic models, the reader can refer to Y. Jacomet \cite{JAC89} and N. Simon \cite{SIM06}.
\\
\\
Recently, in order to modelize perturbations during the elimination processes, stochastic generalizations of (\ref{monocp}) has been studied :
\begin{displaymath}
C_t = \int_{0}^{t}\left(\frac{A_0K_a}{v}e^{-K_a s} - K_eC_s\right)ds +
\int_{0}^{t}\sigma\left(s,C_s\right)dB_s
\textrm{ ; }
t\in [0,T]
\end{displaymath}
where, $B$ is a standard Brownian motion and the stochastic integral is taken in the sense of It\^o. For example, in K. Kalogeropoulos et al. \cite{KDP08} :
\begin{displaymath}
C_t = \int_{0}^{t}\left(\frac{A_0K_a}{v}e^{-K_a s} - K_eC_s\right)ds +
\sigma\int_{0}^{t}C_{s}^{\beta}dB_s
\textrm{ ; }
t\in [0,T]
\end{displaymath}
with $\sigma > 0$ and $\beta\in [0,1]$.
\\
\\
However, these models aren't realistic (cf. M. Delattre and M. Lavielle \cite{DL11}), because the obtained process $C$ is too rough.
\\
\\
Since probabilistic properties of It\^o's integral aren't particularly interesting in that situation, if the drug is administered by rapid injection, $C$ could be the solution of equation (\ref{EDS3}) with $C_0 = A_0/v$, $a = 0$ and $b = K_e$.
\\
\\
In order to bypass the difficulty of the standard Brownian motion's paths roughness, one can take a Gaussian process $W$ satisfying Assumption \ref{HYPW} with $\alpha$ close to $1$. Typically, a fractional Brownian motion $B^H$ with a high Hurst parameter $H$ (cf. simulations below).
\\
\\
Precisely :
\begin{equation}\label{MRmonocp}
C_t = \frac{A_0}{v} - K_e\int_{0}^{t}C_sds +
\sigma\int_{0}^{t}C_{s}^{\beta}dW_s
\end{equation}
where the stochastic integral is taken pathwise, in the sense of Young. Moreover, since $a = 0$, we shown at Section 3 that until it hits zero, the solution of equation (\ref{MRmonocp}) is matching with the process $X$ defined by :
\begin{displaymath}
\forall t\in\mathbb R_+\textrm{, }
X_t =
\left|
\left(\frac{A_0}{v}\right)^{1-\beta} + \tilde W_t\right|^{\gamma + 1}e^{-K_et}
\textrm{ with }
\tilde W_t =
\sigma(1-\beta)
\int_{0}^{t}
e^{K_e(1-\beta)s}dW_s.
\end{displaymath}
It is natural to assume that when the concentration hits 0, the elimination process stops. Then, we put $C = X\mathbf 1_{[0,\tau_{0}^{1}\wedge T[}$ where $T > 0$ is a deterministic fixed time.
\\
\\
For example, let simulate that model with $A_0 = v$, $K_e = 4$, $\sigma = 1$, $\beta = 0.8$ and a fractional Brownian motion $B^H$ with Hurst parameter $H\in\{0.6,0.9\}$ :
\begin{figure}[H]
\centering
\includegraphics[scale = 0.5]{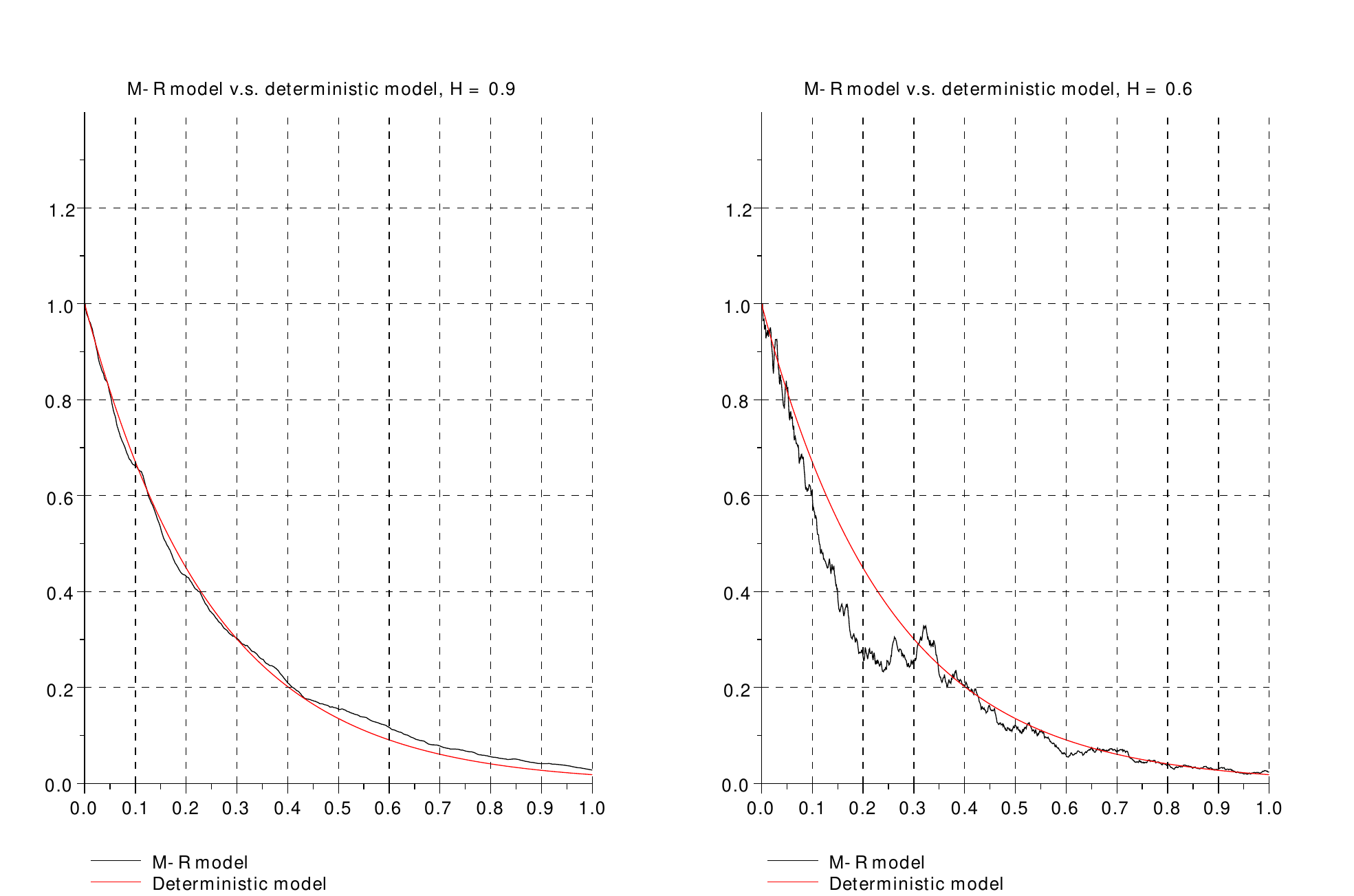}
\caption{GM-R model v.s. deterministic model currently used}
\end{figure}
\noindent
On one hand, remark that the stochastic model (black) keeps the trend of the deterministic model (red). On the other hand, remark that when the Hurst parameter is relatively close to $1$ ($H = 0.9$), perturbations in biological processes are taken in account by $C$, but more realistically than for $H = 0.6$.
\\
\\
In the sequel, we also consider the process $Z = X^{1-\beta}$. Its covariance function is denoted by $c_Z$.
\\
\\
For clinical applications, parameters $K_e$, $\sigma$ and $\beta$ have to be estimated. Consider a dissection $(t_0,\dots,t_n)$ of $[0,T]$ for $n\in\mathbb N^*$. We also put $x_i = X_{t_i}$ and $z_i = Z_{t_i}$ for $i = 0,\dots,n$. The following proposition provides the likelihood function of $(x_1,\dots,x_n)$ which can be approximatively maximized with respect to the parameter $\theta = (K_e,\sigma,\beta)$ by various numerical methods (not studied in this \mbox{paper) :}
%

% Proposition : Likelihood function.

%
\begin{proposition}\label{MLfunction}
Under assumptions \ref{Eass} and \ref{HYPW}, the likelihood function of $(x_1,\dots, x_n)$ is given by :
\begin{displaymath}
L(\theta ; x_1,\dots,x_n) =
\frac{2^n(1-\beta)^n\mathbf 1_{x_1 > 0,\dots, x_n > 0}}
{(2\pi)^{n/2}\sqrt{\left|\det\left[\Gamma(\theta)\right]\right|}}
\exp\left[-\frac{1}{2}\langle\Gamma^{-1}(\theta)U_{n}^{x},U_{n}^{x}\rangle\right]
\prod_{i = 1}^{n}
x_{i}^{-\beta}
\end{displaymath}
where, $\sigma^2(\theta) = \textrm{Var}(z_1,\dots,z_n)$,
\begin{displaymath}
\Gamma(\theta) =
\begin{bmatrix}
\sigma_{1}^{2}(\theta) & \dots & c_Z(t_1,t_n)\\
\vdots & \ddots & \vdots\\
c_Z(t_n,t_1) & \dots & \sigma_{n}^{2}(\theta)
\end{bmatrix}
\textrm{ and }
U_{n}^{x} =
\begin{pmatrix}
x_{1}^{1-\beta} - C_{0}^{1-\beta}e^{-K_e(1-\beta)t_1}\\
\vdots\\
x_{n}^{1-\beta} - C_{0}^{1-\beta}e^{-K_e(1-\beta)t_n}
\end{pmatrix}.
\end{displaymath}
\end{proposition}
%

% Proof.

%
\begin{proof}
Since $\tilde W$ is a centered Gaussian process as a Wiener integral against $W$ ; $(z_1,\dots,z_n)$ is a centered Gaussian vector with covariance matrix $\Gamma(\theta)$. We denote by $f_{1,\dots,n}(\theta ;.)$ the natural density of $(z_1,\dots,z_n)$ with respect to Lebesgue's measure on $(\mathbb R^n,\mathcal B(\mathbb R^n))$.
\\
\\
Consider an arbitrary Borel bounded map $\varphi :\mathbb R^n\rightarrow\mathbb R$. By the transfer theorem :
\begin{eqnarray*}
 \mathbb E[\varphi(x_1,\dots,x_n)] & = &
 \mathbb E[\varphi(|z_1|^{\gamma + 1},\dots,|z_{n}^{\gamma + 1}|)]\\
 & = &
 2^n\int_{\mathbb R_{+}^{n}}
 \varphi(a_{1}^{\gamma + 1},\dots,a_{n}^{\gamma + 1})f_{1,\dots,n}(\theta; a_1,\dots,a_n)da_1\dots da_n
\end{eqnarray*}
by reduction to canonical form of quadratic forms.
\\
\\
Put $u_i = a_{i}^{\gamma + 1}$ for $a_i\in\mathbb R_{+}^{*}$ and $i = 1,\dots,n$. Then,
\begin{displaymath}
(a_1,\dots, a_n) = 
(u_{1}^{\frac{1}{\gamma + 1}},\dots,
u_{n}^{\frac{1}{\gamma + 1}})
\textrm{ and }
\left|J(u_1,\dots,u_n)\right| =
\frac{1}{(\gamma + 1)^n}
\prod_{i = 1}^{n}
u_{i}^{-\frac{\gamma}{\gamma + 1}}
\end{displaymath}
where, $J(u_1,\dots,u_n)$ denotes the Jacobian of :
\begin{displaymath}
(u_1,\dots,u_n)\in (\mathbb R_{+}^{*})^n
\longmapsto
(u_{1}^{\frac{1}{\gamma + 1}},\dots,
u_{n}^{\frac{1}{\gamma + 1}}).
\end{displaymath}
By applying that change of variable :
\begin{eqnarray*}
 \mathbb E[\varphi(x_1,\dots,x_n)] & = &
 \frac{2^n}{(\gamma + 1)^n}
 \int_{\mathbb R_{+}^{n}}du_1\dots du_n
 \varphi(u_1,\dots,u_n)\times\\
 & &
 f_{1,\dots,n}(\theta; u_{1}^{\frac{1}{\gamma + 1}},\dots,u_{n}^{\frac{1}{\gamma + 1}})
 \prod_{i = 1}^{n}
 u_{i}^{-\frac{\gamma}{\gamma + 1}}.
\end{eqnarray*}
Therefore, $\mathbb P_{(x_1,\dots,x_n)}(\theta ; du_1,\dots,du_n) = L(\theta; u_1,\dots,u_n)du_1\dots du_n$ with :
\begin{eqnarray*}
 L(\theta; u_1,\dots,u_n) & = &
 \frac{2^n}{(\gamma + 1)^n}
 f_{1,\dots,n}(\theta; u_{1}^{\frac{1}{\gamma + 1}},\dots,u_{n}^{\frac{1}{\gamma + 1}})
 \prod_{i = 1}^{n}
 u_{i}^{-\frac{\gamma}{\gamma + 1}}
 \mathbf 1_{u_1 > 0,\dots, u_n > 0}\\
 & = &
 \frac{2^n(1 -\beta)^n\mathbf 1_{u_1 > 0,\dots, u_n > 0}}
 {(2\pi)^{n/2}\sqrt{\left|\det\left[\Gamma(\theta)\right]\right|}}
 \exp\left[-\frac{1}{2}\langle\Gamma^{-1}(\theta)U_{n}^{u},U_{n}^{u}\rangle\right]
 \prod_{i = 1}^{n}
 u_{i}^{-\beta}.
\end{eqnarray*}
\end{proof}
\noindent
Finally, consider a random time $\tau\in [0,\tau_{0}^{1}\wedge T]$ and a deterministic function $F :\mathbb R_+\rightarrow\mathbb R$ satisfying the following assumption :
%

% Assumption : Regularity on $F$.

%
\begin{assumption}\label{Fhyp}
The function $F$ belongs to $C^1(\mathbb R_+;\mathbb R)$ and there exists $(K,N)\in\mathbb R_{+}^{*}\times\mathbb N^*$ such that :
\begin{displaymath}
\forall r\in\mathbb R_+\textrm{, }
|F(r)|\leqslant
K(1 + r)^N\textrm{ and }
|\dot F(r)|\leqslant
K(1 + r)^N.
\end{displaymath}
\end{assumption}
\noindent
Let show the existence and compute the sensitivity of $f_{\tau}(x) =\mathbb E[F(C_{\tau}^{x})]$ to variations of the initial concentration $x > 0$ in compartment $E$.
%

% Proposition : Sensitivity with respect to the initial condition.

%
\begin{proposition}\label{PLsin}
Under assumptions \ref{Eass}, \ref{HYPW} and \ref{Fhyp}, the function $f_{\tau}$ is differentiable on $\mathbb R_{+}^{*}$ and,
\begin{displaymath}
\forall x > 0\textrm{, }
\dot f_{\tau}(x) =
x^{-\beta}
\mathbb E\left[e^{-K_e\tau}
\dot F(C_{\tau}^{x})(x^{1-\beta} + \tilde W_{\tau})^{\gamma}\right].
\end{displaymath}
\end{proposition}
%

% Proof.

%
\begin{proof}
First of all, the function $x\in\mathbb R_{+}^{*}\mapsto C_{\tau}^{x}$ is almost surely $C^1$ on $\mathbb R_{+}^{*}$ and,
\begin{displaymath}
\forall x > 0\textrm{, }
\partial_x C_{\tau}^{x} =
x^{-\beta}
\left(x^{1-\beta} +\tilde W_{\tau}\right)^{\gamma}e^{-K_e\tau}.
\end{displaymath}
Consider $x > 0$ and $\varepsilon\in ]0,1]$.
\\
\\
On one hand, since $F$ belongs to $C^1(\mathbb R_+;\mathbb R)$, from Taylor's formula :
\begin{eqnarray*}
 \left|\frac{F(C_{\tau}^{x + \varepsilon}) - F(C_{\tau}^{x})}{\varepsilon}\right| & = &
 \left|\int_{0}^{1}\dot F(C_{\tau}^{x + \theta\varepsilon})\partial_xC_{\tau}^{x + \theta\varepsilon}d\theta\right|\\
 & \leqslant &
 \sup_{\theta\in [0,1]}
 K(1 + \|C^{x + \theta\varepsilon}\|_{\infty;T})^N
 |\partial_xC_{\tau}^{x + \theta\varepsilon}|
\end{eqnarray*}
by Assumption \ref{Fhyp}.
\\
\\
On the other hand, since $\theta,\varepsilon\in [0,1]$ :
\begin{equation}\label{PLsin1}
 \|C^{x + \theta\varepsilon}\|_{\infty;T}\leqslant
 \left[(x + 1)^{1 -\beta} + \|\tilde W\|_{\infty;T}\right]^{\gamma + 1}
\end{equation}
and
\begin{equation}\label{PLsin2}
|\partial_xC_{\tau}^{x + \theta\varepsilon}|
\leqslant
x^{-\beta}\left[(x + 1)^{1-\beta} + \|\tilde W\|_{\infty;T}\right]^{\gamma}.
\end{equation}
By Fernique's theorem, the right hand sides of inequalities (\ref{PLsin1}) and (\ref{PLsin2}) belong to $L^p(\Omega)$ for every $p > 0$. Moreover, these upper-bounds are not depending on $\theta$ and $\varepsilon$.
\\
\\
Therefore, by Lebesgue's theorem, $f_{\tau}$ is derivable at point $x$ and,
\begin{displaymath}
\dot f_{\tau}(x) =
x^{-\beta}
\mathbb E\left[e^{-K_e\tau}
\dot F(C_{\tau}^{x})(x^{1-\beta} + \tilde W_{\tau})^{\gamma}\right].
\end{displaymath}
\end{proof}
\noindent
There is probably many ways to use that result in medical treatments. For example, assume that $f_{\tau}(x)$ modelize a part of patient's therapeutic response to the administered drug. Proposition \ref{PLsin} provides a way to minimize the initial dose for an optimal response.
\\
\\
\textbf{Remarks :}
\begin{enumerate}
 \item By the strong law of large numbers, there exists an almost surely converging estimator for that sensitivity.
 \item For any $x > 0$, one can show the existence of a stochastic process $h^x$ defined on $[0,T]$ such that $\dot f_{\tau}(x) = \mathbb E[F(C_{\tau}^{x})\delta(h^x)]$ where, $\delta$ denotes the divergence operator associated to the Gaussian process $W$. Then, $F$ has not to be derivable anymore by assuming that $F\in L^2(\mathbb R_{+}^{*})$. It is particularly useful if $F$ is not continuous at some points.
 \\
 \\
 We don't develop it in that paper because the Malliavin calculus framework has to be introduced before. To understand that idea, please refer to E. Fourni\'e et al. \cite{FLLLT99} in Brownian motion's case and N. Marie \cite{MARIE11}.
\end{enumerate}
%

% References.

%

%
\end{document}